%% file: main.tex
\documentclass[accepted]{uai2026} 
                        
\usepackage[american]{babel}

\usepackage{natbib} 
\usepackage{mathtools} 
\usepackage{booktabs} 
\usepackage{tikz} 
\usepackage{defs}

\title{Quantized Stochastic Primal–Dual Methods for Distributed Optimization under Relaxed Global Geometry}

\author[1,*]{Susmit~Sarkar}
\author[2,*]{\href{mailto:<abhinavriitb11@gmail.com>?Subject=Your UAI 2026 paper}{Abhinav~Raghuvanshi}{}}
\author[2]{\href{mailto:<chakrabarti.k1@tcs.com>?Subject=Your UAI 2026 paper}{Kushal~Chakrabarti}{}}
\author[1,2]{\href{mailto:<mbaranwal@iitb.ac.in>?Subject=Your UAI 2026 paper}{Mayank~Baranwal}{}}
\affil[1]{%
    Indian Institute of Technology Bombay\\
    Mumbai, Maharashtra, India
}
\affil[2]{%
    Tata Consultancy Services Research\\
    Thane, Maharashtra, India
}
\affil[*]{%
    Equal Contribution
  }
   
\begin{document}
\maketitle

\begin{abstract}

  We study distributed optimization with stochastic gradients and finite-bit communication modeled by random (unbiased) quantization. We propose \texttt{q-PDGD}, a quantized stochastic primal–dual method, and analyze it under relaxed global geometry. Under restricted secant inequality (RSI), a constant step-size yields linear contraction to an explicit neighborhood determined by gradient noise, quantization distortion, and network connectivity, while a diminishing step-size achieves $\bigO(1/k)$ convergence without shared-minimizer assumptions. Under Polyak–Łojasiewicz (PL) inequality, we obtain linear-to-neighborhood convergence in the same stochastic quantized setting. Our results match the best-known centralized stochastic rates in oracle complexity, and are supported by experiments demonstrating the predicted tradeoffs between quantization level, step-size choice, and graph structure.
\end{abstract}

\input{sections/intro}

\input{sections/related_work}
\input{sections/notations}

\input{sections/main_result}

\input{sections/experiments}
\input{sections/conclusion}

\clearpage

\bibliography{uai2026-template}

\newpage

\include{sections/supplementary}

\end{document}

%% file: sections/intro.tex
\section{Introduction}\label{sec:intro}
Distributed optimization is a central primitive in large-scale learning and networked control, enabling agents to cooperatively minimize a global objective while keeping data and computation local~\citep{yang2019survey}. It underpins federated and data-parallel learning~\citep{yu2024provable}, large-scale empirical risk minimization~\citep{jahani2020efficient}, distributed estimation in sensor/IoT networks~\citep{chong2018review}, and network resource allocation/control~\citep{nedic2018distributed}. In many such systems, communication is the dominant bottleneck, motivating finite-bit (\emph{quantized}) protocols.

Quantization has been studied extensively, with early control-oriented approaches using deterministic uniform quantizers (often with saturation) in consensus/ADMM/primal-dual message passing, which introduces bias and steady-state errors that must be controlled~\citep{zhu2016quantized}. Modern learning further requires \emph{stochastic gradients} (mini-batches/noisy measurements), creating a second persistent error source. Recent results show that sharp sublinear rates can still be attainable under random quantization in distributed gradient methods with appropriate conditions and stepsizes; in particular, \cite{dutta20241} establishes $\bigO(1/k)$ convergence under specific assumptions. However, existing guarantees for stochastic quantized distributed optimization often emphasize server-based architectures or primal consensus/gradient-tracking schemes, and typically yield neighborhood convergence under constant step-sizes or rely on restrictive structural assumptions (e.g., shared minimizers) for exact convergence.

Motivated by these gaps, we study \emph{primal-dual} methods for decentralized stochastic optimization under \emph{random (unbiased) quantized} communication. We develop and analyze \texttt{q-PDGD}, providing convergence guarantees under relaxed global geometry (RSI/PL) and explicit tradeoffs among step-size, quantization resolution, network connectivity, and stochastic noise. Our main contributions are:
\begin{itemize}
    \item \textbf{Stochastic quantized PDGD under RSI.} We establish convergence of a stochastic primal-dual gradient method with unbiased random quantization (\texttt{q-PDGD}) under RSI, without strong convexity or shared-minimizer assumptions.
    \item \textbf{Rates with constant and diminishing step-sizes.} For constant step-sizes, we prove linear contraction to an explicit noise/quantization/network-dependent neighborhood; for diminishing step-sizes, we obtain $\bigO(1/k)$ rates in the strongly-convex/RSI regime.
    \item \textbf{PL regime and experiments.} Under PL, we show linear-to-neighborhood convergence under stochasticity and quantization, and we empirically validate the predicted step-size/bit-rate tradeoffs and the linear-to-neighborhood vs.\ sublinear behavior.
\end{itemize}

%% file: sections/related_work.tex
\section{Related Work}\label{sec:RelatedWork}
Finite-bit communication in distributed optimization has been studied from classic \emph{deterministic} quantizers (often with clipping) to modern stochastic/structured compressors. Deterministic approaches appear in quantized consensus and ADMM-style primal--dual message passing, which characterize steady-state/consensus errors under finite-capacity links~\citep{zhu2016quantized}. Beyond fixed quantizers, algorithm-quantizer co-design yields stronger rates: \cite{pu2016quantization} develop iteratively-refining schemes for distributed proximal-gradient-type methods and show linear decay of quantization error and linear convergence under suitable interval conditions, while \cite{rikos2023distributed} combine gradient descent with a finite-time quantized-consensus primitive over directed graphs and obtain linear-to-neighborhood convergence for strongly convex objectives.

A complementary abstraction models quantization as a \emph{compression operator} with either unbiasedness and bounded second moment or contractive/relative-error properties, unifying \texttt{QSGD}~\citep{alistarh2017qsgd}, \texttt{TernGrad}~\citep{wen2017terngrad}, and \texttt{signSGD}~\citep{bernstein2018signsgd}. To mitigate bias and stabilize learning, \emph{error-feedback} methods such as \texttt{EF21} provide clean guarantees, including improved rates for smooth nonconvex objectives and linear rates under PL-type conditions~\citep{richtarik2021ef21}. Related ideas compress gradient differences (\texttt{DIANA}) for sharper guarantees~\citep{mishchenko2025distributed} and use double quantization to reduce communication~\citep{yu2019double}.

Rates also depend on architecture: server-based methods often admit explicit accuracy--communication tradeoffs under unbiased/difference compression~\citep{alistarh2017qsgd,mishchenko2025distributed,yu2019double}, whereas decentralized methods must control consensus/tracking error in addition to compression noise; \texttt{CHOCO-SGD} analyzes decentralized learning under broad compressor models and general nonconvex objectives~\citep{koloskova2019decentralized}. For constrained network optimization, quantized primal--dual dynamics with adaptive encoding/decoding can achieve linear convergence to the exact optimum for convex global costs while relating bandwidth to rate~\citep{chen2023quantized}. Finally, recent stochastic-approximation analyses of random quantization identify conditions/stepsizes achieving sharp sublinear rates, including $\bigO(1/k)$~\citep{dutta20241}.

In contrast, we develop and analyze \texttt{q-PDGD} under RSI/PL with stochastic gradients and unbiased quantization, yielding explicit neighborhood and rate tradeoffs.

%% file: sections/notations.tex
\section{Notations \& Preliminaries}
\label{sec:notations}

For any positive integer $n$, let $[n]$ denote the set $\{1, \dots, n\}$.
We denote the concatenated column vector of vectors $z_i \in \mathbb{R}^{p_i}$ for $i \in [k]$ by $\text{col}(z_1, \dots, z_k)$.
Let $\mathbf{1}_n$ and $\mathbf{0}_n$ denote the column vectors of all ones and all zeros of dimension $n$, respectively.
$\bI_n$ denotes the $n$-dimensional identity matrix.
Given a vector $[x_1, \dots, x_n]^\top \in \mathbb{R}^n$, $\text{diag}([x_1, \dots, x_n])$ represents a diagonal matrix with the $i$-th diagonal element being $x_i$.
The notation $\boldsymbol{A} \otimes \boldsymbol{B}$ denotes the Kronecker product of matrices $\boldsymbol{A}$ and $\boldsymbol{B}$.

For a matrix $\boldsymbol{A}$, $\text{rank}(\boldsymbol{A})$, $\text{image}(\boldsymbol{A})$, and $\text{null}(\boldsymbol{A})$ denote its rank, image space, and null space, respectively.
For two symmetric matrices $\boldsymbol{M}$ and $\boldsymbol{N}$, the inequality $\boldsymbol{M} \succeq \boldsymbol{N}$ indicates that $\boldsymbol{M} - \boldsymbol{N}$ is positive semi-definite.
The spectral radius of a matrix is denoted by $\rho(\cdot)$, and $\rho_2(\cdot)$ indicates the smallest non-zero eigenvalue for matrices having positive eigenvalues.
$\|\cdot\|$ represents the standard Euclidean norm for vectors or the induced 2-norm for matrices.
For a given positive semi-definite matrix $\boldsymbol{A}$, the weighted norm is denoted by $\|\bx\|_{\boldsymbol{A}} = \sqrt{\bx^\top \boldsymbol{A} \bx}$.
For a differentiable function $g$, $\nabla g$ denotes the gradient of $g$.
Finally, $\Proj_{S}(x) = \arg\min_{y \in S} \|x-y\|^2$ denotes the projection of a vector $x \in \R^n$ onto the set $S \subseteq \R^n$.

\begin{definition} \label{def:rsi}
A differentiable function $f: \R^n \to \R$ satisfies the Restricted Secant Inequality (RSI) with constant $\nu > 0$ if $(\nabla f(x) - \nabla f(\Proj_{X^*}(x))^{\top} (x-\Proj_{X^*}(x)) \geq \nu \norm{x-\Proj_{X^*}(x)}^2 \, \forall \bx \in \R^n$, where $X^*$ is the set of global minimizers of $f$~\citep{shi2015extra, yi2020exponential}.
\end{definition}

\begin{definition} \label{def:PL}
A differentiable function $f: \R^n \to \R$ satisfies the Polyak--Łojasiewicz (PL) inequality with constant $\nu > 0$ if $\frac{1}{2}\|\nabla f(x)\|^2 \ge \nu \bigl(f(x) - f^*\bigr),
\quad \forall x \in \mathbb{R}^n$, where $f^* = \min_{x \in \R^n} f(x)$~\citep{karimi2016linear, yi2020linear}.
\end{definition}

The PL inequality is the weakest condition among others, such as strong convexity, essential strong convexity, weak strong convexity, and restricted secant inequality, that leads to {\em linear} convergence of deterministic centralized gradient-based methods~\citep{karimi2016linear}. Both RSI and PL are weaker than strong convexity. For example, all linear least-squares problem $f(x) = \|Ax-b\|^2$ satisfy PL, but not necessarily strongly convex. Non-convex functions like $f(x) = x^2 + 3 \sin^2(x)$, which satisfy PL, demonstrate the condition's generality beyond strong convexity. Similarly, all quasi-strongly convex functions or composition of strongly convex function with linear map plus a linear term satisfy RSI. One such example function is presented later in Section~\ref{sec:exp}. Note that, neither of RSI and PL conditions require the function to be convex.

\begin{definition} \label{def:graph}
Let $\mathcal{G} = (\mathcal{V}, \mathcal{E}, A)$ be a weighted undirected graph,
where $\mathcal{V} = [n]$ is the set of nodes, 
$\mathcal{E} \subseteq \mathcal{V} \times \mathcal{V}$ is the set of edges,
and $A = A^\top = (a_{ij})$ is the weighted adjacency matrix with 
$a_{ij} \ge 0$ and $a_{ii} = 0$. Nodes $i$ and $j$ can communicate if $(i,j) \in \mathcal{E}$, i.e., $a_{ij} > 0$.
The neighbor set and weighted degree of node $i$ are defined as
$\mathcal{N}_i = \{ j \in [n] : a_{ij} > 0 \}, \,
\deg_i = \sum_{j=1}^n a_{ij}$.
Let $D = \mathrm{diag}([\deg_1, \dots, \deg_n])$ denote the degree matrix. The weighted Laplacian matrix is defined as $L = D - A$.
The graph $\mathcal{G}$ is said to be connected if there exists a path
between any pair of nodes.
\end{definition}

%% file: sections/main_result.tex
\section{Problem Setup and the \texttt{\lowercase{q}-PDGD} Algorithm}
\label{sec:problem}


Consider a network of $n$ agents, where each agent $i \in \{1, \dots, n\}$ has a local cost function $f_i: \R^d \to \R$. The agents collaborate to solve the following distributed optimization problem:
\begin{equation}
    \min_{x \in \R^d} \; f(x) \coloneqq \sum_{i=1}^n f_i(x).
    \label{eq:global_problem}
\end{equation}
Let $x^*$ denote an optimal solution of \eqref{eq:global_problem}, and let $X^* \coloneqq \arg\min_{x \in \R^d} f(x)$ be the corresponding optimal set. Each agent maintains a local copy $x_i \in \R^d$ of the solution $x^*$.
For simplicity, we define the stacked variables and function:
\begin{equation}
    \bx \coloneqq \text{col}(x_1, \dots, x_n) \in \R^{nd}, \quad F(\bx) \coloneqq \sum_{i=1}^n f_i(x_i).
\end{equation}
The optimal set in the stacked space is $\calX \coloneqq \{\mathbf{1}_n \otimes x^* : x^* \in X^*\}$. Further, we let $f^* \coloneqq \min_{x \in \mathbb{R}^d} f(x)$
denote the optimal value of the global objective.

\paragraph{Network Model.}
The communication among agents is described by an undirected weighted graph $\mathcal{G} = (\mathcal{V}, \mathcal{E}, A)$ with $\mathcal{V} = \{1, \dots, n\}$ as defined by Definition~\ref{def:graph}. The graph is assumed to be connected. Let $L$ denote the weighted Laplacian matrix associated with $\mathcal{G}$, and define the augmented Laplacian $\bL \coloneqq L \otimes \bI_d$.

\paragraph{Random Quantization Model.}
Each agent $i$ quantizes its decision variable $x_i$ using a quantization function $Q(\cdot)$, resulting in the quantized value $Q(x_i)$. We consider a quantization scheme characterized by the hyperparameters $\{l, u, b\}$, where $[l, u]$ denotes the quantization domain and $b$ denotes the number of bits available for communication.
We partition the interval $[l, u]$ into $B = 2^b - 1$ equal-length bins.\footnote{Note that $B=2^b-1$ bins result in $B+1 = 2^b$ quantization levels, which can be uniquely represented by a $b$-bit codeword.} The length of each bin, referred to as the \textit{quantization resolution}, is given by $\Delta \coloneqq \frac{u-l}{B}$. Let $\{\tau_m\}_{m=1}^{B+1}$ denote the quantization levels (endpoints), where $\tau_1 = l$ and $\tau_{B+1} = u$.
Given a value $x \in [l, u]$, we identify the interval such that $x \in [\tau_i, \tau_{i+1})$. We define the relative location of $x$ within the bin as $p = \frac{x - \tau_i}{\Delta}$. The stochastic quantizer $Q: [l, u] \to \{\tau_1, \ldots, \tau_{B+1}\}$ is defined as:
\begin{equation}
    Q(x) = \begin{cases}
    \tau_i, & \text{with probability } 1-p, \\
    \tau_{i+1}, & \text{with probability } p.
    \end{cases}
\end{equation}
In this setting, at every iteration, agent $i$ transmits the quantized value $Q(x_i)$, with its neighboring agents. For a vector $x\in \R^d$, quantization is applied element-wise. We refer to $\varepsilon(x) \coloneqq Q(x) - x$ as the \textit{quantization error}.




\begin{lemma}[\citep{reisizadeh2018quantized}] \label{lem:quant}
For any input $z \in [l, u]$, the random quantization scheme $Q(\cdot)$ is unbiased and has bounded variance. Specifically, it satisfies the following properties:
\begin{align}
\mathbb{E}[Q(z) \mid z] &= z, \, \qquad
\mathbb{E}[|Q(z) - z|^2 \mid z] \le \frac{\Delta^2}{4}. \label{eq:quant_variance}
\end{align}
\end{lemma}

\noindent \textbf{Quantization Noise.}
We define the quantization noise for agent $i$ at iteration $k$ as $\varepsilon_i(k) \coloneqq Q(x_i(k)) - x_i(k)$. We let $\calF_k$ represent the filtration that contains the history of all variables generated by an algorithm up to iteration $k$. Since the quantization is performed element-wise on the vector $x_i(k) \in \mathbb{R}^d$, the properties established in Lemma~\ref{lem:quant} imply that the noise is zero-mean and has bounded variance conditioned on $\calF_k$:
\begin{align}
    \mathbb{E}[\varepsilon_i(k) \mid \calF_k] = 0, \qquad
    \mathbb{E}[\|\varepsilon_i(k)\|^2 \mid \calF_k] \le \frac{d \Delta^2}{4}. \label{eq:noise_bound}
\end{align}

\subsection{Assumptions}
\label{sub:ass}

We adopt the following assumptions for the analysis.

\begin{assumption}
\label{ass:diff}
The set of optimal solutions $X^* = \arg\min_{x\in\R^d} f(x)$
is nonempty and convex.
\end{assumption}

\begin{assumption}
\label{ass:smooth}
Each $f_i:\R^d \to \R$ is continuously differentiable and has a globally Lipschitz continuous gradient with constant $L_{f_i} > 0$, i.e.,
$\norm{\nabla f_i(a) - \nabla f_i(b)} \le L_{f_i}\norm{a-b}, \quad \forall a,b \in \R^d$.

\end{assumption}

\begin{assumption}
\label{ass:rsi}
The global objective function $f(x)$ satisfies the Restricted Secant Inequality (RSI) condition with constant $\nu > 0$, i.e.,
\begin{multline}
    (\nabla f(x) - \nabla f(\Proj_{X^*}(x)))^\top (x - \Proj_{X^*}(x)) \\
    \ge \nu \norm{x - \Proj_{X^*}(x)}^2, \quad \forall x \in \R^d.
\end{multline}
\end{assumption}

\begin{assumption}
\label{ass:PL}
The global objective function $f(x)$ satisfies the Polyak--Łojasiewicz (PL) inequality with constant $\nu > 0$, i.e.,
\begin{align}
\frac{1}{2}\|\nabla f(x)\|^2 \ge \nu \bigl(f(x) - f^*\bigr),
\quad \forall x \in \mathbb{R}^d.
\end{align}
\end{assumption}

\begin{assumption}
\label{ass:singleton}
The set $\{\nabla F(\bx) : \bx \in \calX\}$ is a singleton. 
\end{assumption}

\begin{assumption}
\label{ass:connectivity}
The communication graph $\mathcal{G}$ between agents is undirected and connected. 
\end{assumption}

\begin{assumption}
\label{ass:stoch_grad}
Each agent has access to an unbiased stochastic gradient $g_i(x_i(k), \zeta_i(k)) = \nabla f_i(x_i(k)) + \zeta_i(k)$ with bounded variance:
\begin{equation}
    \E[\zeta_i(k) \mid \calF_k] = 0, \qquad \E[\norm{\zeta_i(k)}^2 \mid \calF_k] \le \sigma^2.
\end{equation}
\end{assumption}


From Assumption~\ref{ass:smooth}, we define the maximum Lipschitz constant across all agents as $L_f \coloneqq \max_{i \in [n]} L_{f_i}$. 
Assumption~\ref{ass:singleton} implies that the aggregate gradient is invariant across the optimal set $\calX$; specifically, the sum of local gradients at any optimal solution is unique. 
Finally, Assumption~\ref{ass:connectivity} ensures that the Laplacian $L$ has a simple zero eigenvalue. Consequently, its second-smallest eigenvalue $\rho_2(L) > 0$.

Compared to the analysis of Distributed Consensus Stochastic Gradient (\texttt{DCSG}) method under the same quantization model from~\cite{dutta20241}, we do not assume $X^*$ being singleton, and crucially, we do not assume that a global minimizer $x^* \in X^*$ also minimizes each $f_i$. Furthermore, the RSI condition in Assumption~\ref{ass:rsi} and PL condition in Assumption~\ref{ass:PL} are both weaker than the strong convexity of global objective in~\cite{dutta20241}.

Now, we review below in Lemma~\ref{lem:rsi_augmented} a prior result from~\cite{yi2020exponential}, an implication of the RSI condition, that is  pivotal for our key results presented in the next section.

\begin{lemma}[\citep{yi2020exponential}]
\label{lem:rsi_augmented}
Let Assumptions \ref{ass:diff}--\ref{ass:rsi} hold. If $\alpha > \frac{2nL_f^2 + \nu L_f}{\nu\rho_2(L)}$, then for all $\bx \in \R^{nd}$,
\begin{multline}
    (\nablaF(\bx) - \nablaF(\Proj_{\calX}(\bx)))^\top (\bx - \Proj_{\calX}(\bx)) + \alpha \bx^\top \bL \bx \\
    \ge \nu_1 \norm{\bx - \Proj_{\calX}(\bx)}^2,
\end{multline}
where $\nu_1 = \min \left\{ \frac{\nu}{2n}, \alpha\rho_2(L) - \frac{2nL_f^2 + \nu L_f}{\nu} \right\} > 0$.
\end{lemma}

\subsection{Stochastic Quantized PDGD Algorithm}
\label{sub:algo}

As discussed in Section~\ref{sec:intro}, primal-dual gradient descent (PDGD) offers an effective framework for decentralized optimization. We now introduce \texttt{q-PDGD}, a quantized stochastic extension of PDGD whose deterministic roots trace back to~\citep{kia2015distributed,yi2020exponential}. Our focus is the practically relevant regime with \emph{simultaneous} inexactness from stochastic gradients and random (unbiased) quantized communication, as formalized in Section~\ref{sec:problem}.

At each iteration $k \geq 0$, each agent $i$ updates its primal and dual variables as
\begin{tcolorbox}[colback=white, colframe=black, sharp corners]
\begin{align}
    x_i(k+1) &= x_i(k) \notag\\
    &\!\!\!\!\!\!\!\!\!\!\!\!\! - h \left(\!\alpha \sum_{j=1}^{n} L_{ij} Q(x_j(k)) + \beta v_i(k) + g_i(x_i(k))\!\right) \notag \\
    v_i(k+1) &= v_i(k) + h \beta \sum_{j=1}^{n} L_{ij} Q(x_j(k)). \tag{\texttt{q-PDGD}} \label{eq:update_law}
\end{align}
\end{tcolorbox}
Here, $x_i(0)$ are arbitrarily initialized, whereas $v_i(0) = \mathbf{0}_d$. The hyperparameters $\alpha, \beta$ and stepsize $h$ are positive-valued. Note that only the primal quantized variable $Q(x_j(k))$ is communicated between neighbors.

\begin{remark}
    Using Definition~\ref{def:graph}, the primal update in Algorithm~\eqref{eq:update_law} can be rewritten as 
    \begin{align}
        x_i(k+1) = & x_i(k) - h \alpha \left(Q(x_i(k) - \sum_{j=1}^{n} a_{ij} Q(x_j(k))\right) \notag \\
        & - h \beta v_i(k) - h g_i(x_i(k)) \notag \\
        = & (1-h\alpha) x_i(k) + h\alpha \sum_{j=1}^{n} a_{ij} Q(x_j(k)) \notag \\
        & - h \beta v_i(k) - h g_i(x_i(k)).
    \end{align}
Upon comparing it with the \texttt{DCSG} algorithm from~\cite{dutta20241}
\begin{align}
    x_i(k+1) = & (1-\beta_{k})x_i(k) + \beta_{k} \sum_{j=1}^{n} a_{ij} Q(x_j(k)) \notag\\
    & -\alpha_{k} g_i(x_i(k)), \label{eq:update_law_DCSG} 
\end{align}
we note that $h \alpha$ in~\eqref{eq:update_law} is analogous to $\beta_k$ in~\texttt{DCSG}. The \texttt{DCSG} $\beta$-mixing term is equivalent to a Laplacian-driven consensus step in~\eqref{eq:update_law}, with $\beta_k$ corresponding to the effective consensus step size $h \alpha$. Although both algorithms implement consensus + local innovation dynamics under quantized communication, the additional dual variable $v_i$ in~\eqref{eq:update_law} serves as a disagreement-tracking mechanism. Unlike the \texttt{DCSG} $\beta_k$-mixing, which explicitly controls convex averaging, $v_i$ accumulates consensus errors and feeds them back into the primal dynamics. This feedback introduces a correction that mitigates persistent disagreement caused by mixing of quantized primal variables.
\end{remark}

Let $\varepsilon(k)$ denote the stacked quantization noise vector such that $Q(\bx(k)) = \bx(k) + \varepsilon(k)$. Similarly, let $\zeta(k) = \text{col}(\zeta_1(k), \dots, \zeta_n(k))$ denote the stacked stochastic gradients. Then, the algorithm~\eqref{eq:update_law} can be written in compact stacked form as
\begin{subequations}
\label{eq:update_law_stacked}
\begin{align}
    & \bx(k+1) = \bx(k) \notag\\
    &\quad \!\!- h \left(\!\alpha \bL \bx(k) + \beta \bv(k) + \nabla \!F(\bx(k)) + \alpha \bL \varepsilon(k) + \zeta(k)\!\right),\\
    & \bv(k+1) = \bv(k) + h \beta \left( \bL \bx(k) + \bL \varepsilon(k) \right).
\end{align}
\end{subequations}
Upon defining the state $\bz(k) = [\bx(k)^\top, \bv(k)^\top]^\top$, the nominal field $G(\bz)$, and total noise $\xi(k)$:
\begin{equation}
    G(\bz) \! = \! \begin{bmatrix} \alpha \bL\bx + \beta \bv + \nablaF(\bx) \\ -\beta \bL\bx \end{bmatrix},
    \xi(k) \! = \! \begin{bmatrix} \alpha \bL\varepsilon(k) + \zeta(k) \\ -\beta \bL\varepsilon(k) \end{bmatrix},
\end{equation}
the algorithm dynamics~\eqref{eq:update_law_stacked} becomes
\begin{equation}
    \bz(k+1) = \bz(k) - h(G(\bz(k)) + \xi(k)). \label{eq:z_dyn}
\end{equation}

\section{Key Results and Convergence Analysis of \texttt{\MakeLowercase{q}-PDGD}}
\label{sec:conv}

In this section, we present the convergence guarantees of the \texttt{q-PDGD} algorithm described in Section~\ref{sub:algo}. We first present below the guarantees of~\eqref{eq:update_law} under RSI (Assumption~\ref{ass:rsi}), followed by the guarantee under PL (Assumption~\ref{ass:PL}).

To present the results under RSI, we introduce the following additional notations:
\begin{itemize}
    \item $\eta = \sqrt{2} \max\left\{\frac{2\epsilon_1}{\rho_2(L)} + \alpha + 1, 4\epsilon_1 + 1\right\} > 0$,
    \item $\epsilon_1 = \max \left\{ \frac{1}{\nu_1} \left( \frac{L_f^2}{2\beta} + \beta \rho(L) \right), \frac{1}{2} \right\}$,
    \item $\epsilon_2 = \min \left\{ \frac{\beta}{2}, \epsilon_1 \nu_1 \right\}$,
    \item $\epsilon_3 = \max \left( \epsilon_1 + \frac{1}{2}, \frac{\epsilon_1}{\rho_2(L)} + \frac{\alpha}{2\beta} + \frac{1}{2} \right)$,
    \item $\epsilon_4 = \min \left( \epsilon_1 - \frac{1}{2}, \frac{\alpha}{2\beta} - \frac{1}{2} \right)$,
    \item $\epsilon_5 = \max\{\beta^2 \rho^2(L) + 3\alpha^2\rho^2(L) + 3L_f^2, 3\beta^2\}$.
    \item $A = \Bigg(\dfrac{\epsilon_3\eta h}{\,2\epsilon_2\epsilon_4-h\eta\epsilon_3\epsilon_5\,}\Bigg)\, \dfrac{nd}{4}\,(2\alpha^2+\beta^2)\,\rho^2(L)$
    \item $B = \Bigg(\dfrac{\epsilon_3\eta h}{\,2\epsilon_2\epsilon_4-h\eta\epsilon_3\epsilon_5\,}\Bigg)\, 2n$
\end{itemize}

Throughout this section, the expectation $\E\left[\cdot\right]$ is taken with respect to the randomness in both the stochastic gradients and the quantization distortion of the \texttt{q-PDGD} algorithm, and conditioned on the history $\calF_k$.

The following result holds for fixed stepsize $h$. 
\begin{theorem}
\label{thm:sq_pdgd_convergence}
Let Assumptions \ref{ass:diff}-\ref{ass:rsi}, \ref{ass:singleton}-\ref{ass:stoch_grad} hold. Consider Algorithm~\eqref{eq:update_law} with parameters $\alpha > \frac{2nL_f^2 + \nu L_f}{\nu\rho_2(L)}$, $\beta > 0$, and $0 < h < \min \left\{\frac{2 \epsilon_2 \epsilon_4}{\eta \epsilon_3 \epsilon_5}, \frac{2\epsilon_3}{\epsilon_2}\right\}$. Then there exist $\rho\coloneqq 1 - \frac{h(2\epsilon_2\epsilon_4 - h\eta\epsilon_3\epsilon_5)}{4\epsilon_3\epsilon_4} \in (0,1)$ and constant $C > 0$ such that
\begin{align}\label{eq:theorem1_bound}
    \E[\|\bx(k)-\Proj_{\calX}(\bx(k))\|]
    \leq C\rho^k + \sqrt{A\Delta^2 + B\sigma^2}.
\end{align}
\end{theorem}

Theorem~\ref{thm:sq_pdgd_convergence} implies that under RSI condition, for small enough but fixed step-size $h>0$, the agents' estimates in the \texttt{q-PDGD} algorithm converge {\em linearly} in expectation to a neighborhood of some optimal solution in $x^* \in X^*$ with rate at least $\rho$. This neighborhood of $x^*$ is proportional to the variance of the combined noise source. Note that exact consensus is not guaranteed above for a fixed $h$. However, both exact convergence to an optimal solution and consensus can be guaranteed for a diminishing step-size condition, which is formally presented in the next result.

\begin{remark}[\textbf{Network scaling of $A$ and $B$}]\label{rem:AB_network}
The constants $A, B$ characterize the asymptotic error neighborhood in~\eqref{eq:theorem1_bound} arising from quantization noise ($\Delta$) and stochastic gradient noise ($\sigma$), respectively. Appendix~\ref{app:AB_scaling} elaborates how these terms are governed by the network spectral quantities $(\rho(L),\rho_2(L))$ and the network scale $n$ while the problem parameters $(L_f,\nu)$ are $\bigo(1)$.
Particularly, in the small stepsize regime, 
$A=\bigo\!\left(h\,d\,\frac{n^4\rho^4(L)}{\rho_2^3(L)}\right)$ and
$B=\bigo\!\left(h\,\frac{n^2\rho^2(L)}{\rho_2(L)}\right)$,
revealing strong sensitivity to the {\em network condition number} 
$\kappa_L \coloneqq \rho(L)/\rho_2(L)$. Poor network condition (large $\kappa_L$) or large networks therefore amplify residue unless $h$ is reduced.
At the maximal stable stepsize 
$h=\bigo\!\Big(\frac{\epsilon_2\,\epsilon_4}{\eta\,\epsilon_3\,\epsilon_5}\Big)
=\bigo\!\Big(\frac{\rho_2^3(L)}{n^3\rho^4(L)}\Big)$, the quantization-induced error is
$A=\bigo(d\,n)$ and stochastic gradient-induced error $B=\bigo\!\left(\frac{\rho_2^2(L)}{n\,\rho^2(L)}\right)$.
This behavior can be interpreted as follows. Quantization error propagates through the consensus step itself. Pushing $h$ to 
its stability limit shrinks the stepsize in the exact inverse order of the dominant network amplification term, canceling spectral effects in $A$. However, since quantization noise is injected at every node, the dynamics accumulate $n$ noise sources in the mixing step, yielding the unavoidable growth of $A$ with $n$. 
On the other hand, gradient noise enters locally and are subsequently averaged by the coupled dynamics. In this case, the Laplacian acts as a variance-reduction mechanism, explaining why $B$ increases with stronger connectivity (smaller $\kappa_L$); stronger connectivity accelerates 
mixing, which improves contraction but also spreads stochastic gradient noise 
more efficiently across the network. At the same time, larger network promotes variance averaging, yielding a decay of $B$ with $n$.
\end{remark}

\begin{theorem}
\label{thm:diminishing_step}
Let Assumptions \ref{ass:diff}-\ref{ass:rsi}, \ref{ass:singleton}-\ref{ass:stoch_grad} hold.
Consider Algorithm~\eqref{eq:update_law} with parameters $\alpha > \frac{2nL_f^2 + \nu L_f}{\nu\rho_2(L)}$, $\beta > 0$, and stepsize $h_k = \frac{\gamma}{k+1}$ with $\gamma > \frac{2\epsilon_3}{\epsilon_2}$. Then there exists a constant $C > 0$ such that for sufficiently large $k$,
\begin{equation} \label{eq:err_dim}
    \E\left[\norm{\bx(k) - \Proj_{\calX}(\bx(k))}^2\right] \le \frac{C}{k}.
\end{equation}
\end{theorem}
Eq~\eqref{eq:err_dim} implies that $\lim_{k \to \infty} \E\left[\norm{\bx(k) - \Proj_{\calX}(\bx(k))}^2\right] = 0$, which means $\E\left[\norm{x_i(k) - \Proj_{X^*}(x_i(k))}\right] \to 0$ and $\E\left[\norm{x_i(k) - x_j(k)}\right] \to 0$ for all $i,j$. Thus, under RSI condition, with a diminishing stepsize, the agents' estimates in the \texttt{q-PDGD} algorithm achieve consensus to some optimal solution $x^* \in X^*$ in expectation at a rate of $\mathcal{O}(1/k)$, matching the best-known centralized stochastic rates in oracle complexity.

While RSI in Assumption~\ref{ass:rsi} relaxes the classical strong convexity condition, an even weaker relaxation is PL in Assumption~\ref{ass:PL}. We present below the convergence guarantee of Algorithm~\eqref{eq:update_law} with fixed stepsize under the weaker PL condition. To present this result, we need the following notation.
\begin{align*}
\epsilon
&\coloneqq \frac{h(\epsilon_1-\eta\epsilon_2)}{\epsilon_5}, \\
\displaybreak[2]  
\kappa_1
&\coloneqq \frac{1}{\rho_2(L)}\Bigl(4+\frac{3}{2}L_f^2\Bigr), \,
\kappa_2 > 1, \,
\kappa_3
\coloneqq \frac{\kappa_1}{\kappa_2-1}, \\
\kappa_4
&\coloneqq \frac{1}{4}\Bigl(3+\bigl(9+8\kappa_2+\frac{8}{\rho_2(L)}\bigr)^{1/2}\Bigr), \\
\kappa_5
&\coloneqq 2\Bigl(\!\kappa_2+\frac{1}{\rho_2(L)}\!\Bigr)L_f^2
+2\Bigl(\!\bigl(\kappa_2+\frac{1}{\rho_2(L)}\bigr)^2L_f^4+L_f^2\!\Bigr)^{\!1/2},\\
\epsilon_2
&\coloneqq \max\Bigl\{\!\beta^2\rho(L)
+(2\alpha^2\!+\!\beta^2)\rho^2(L)
+\frac{5}{2}L_f^2,\;
2\beta^2\!+\!\frac12\Bigr\}, \\
\epsilon_3
&\coloneqq \frac14-\frac{1}{2\beta}\Bigl(\frac1\beta+\frac{1}{\rho_2(L)}+\frac{\alpha}{\beta}\Bigr)L_f^2, \\
\epsilon_4
&\coloneqq \frac{1}{\beta^2}\Bigl(1+\frac{3}{2\rho_2(L)}+\frac{3\alpha}{2\beta}\Bigr)L_f^2
+\frac{L_f(1+L_f)}{2}, \\
\epsilon_5
&\coloneqq \frac{\alpha+\beta}{2\beta}+\frac{1}{2\rho_2(L)}, \\
\epsilon_1
&\coloneqq \min\Bigl\{1,\frac{\nu}{2n}, 2\sqrt{\epsilon_{2} \epsilon_{5}}\Bigr\}.
\end{align*}
We let $\bar{x}(k) = \frac{1}{n}(\boldsymbol{1}_n^\top \otimes \boldsymbol{I}_d)\boldsymbol{x}(k)$ denote the mean of agents' primal variables and $\bar{\boldsymbol{x}}(k) = \boldsymbol{1}_n \otimes \bar{x}(k)$ its stacked value.



\begin{theorem}
\label{thm:pl_diminishing_step}
Let Assumptions \ref{ass:diff}-\ref{ass:smooth}, \ref{ass:PL}, \ref{ass:connectivity}-\ref{ass:stoch_grad} hold. Consider Algorithm~\eqref{eq:update_law} with parameters $\beta + \kappa_1 \le \alpha \le \kappa_2 \beta$, $\beta \ge \max\{\kappa_3,\kappa_4,\kappa_5\}$, and $0 < h < \min\Bigl\{\frac{\epsilon_1}{\epsilon_2}, \frac{\epsilon_3}{\epsilon_4}\Bigr\}$. Then there exists a constant $C_{err} > 0$ such that for all agent $i \in [n]$,
\[
\limsup_{k \to \infty} \left(\E[\|x_{i,k}-\Proj_{X^*}(\bar{x}_{k})\|^{2}] + \E[f(\bar{x}_{k})-f^{*}]\right) \le C_{err}.
\]
\end{theorem}

Consequently, under PL inequality, iterates of the \texttt{q-PDGD} algorithm achieve mean-square convergence in expectation to a neighborhood of the optimum. In particular, both the consensus error and objective sub-optimality gap remain uniformly bounded asymptotically.




\subsection{Proof Sketches of Key Results}

\begin{proof}[Proof Sketch of Theorem~\ref{thm:sq_pdgd_convergence}]
The proof relies on constructing a composite Lyapunov function $V(\bz)$ for the augmented state $\bz = [\bx^\top, \bv^\top]^\top$, defined as $V = V_2 + V_3$. Here, $V_2$ captures the quadratic error weighted by the spectral properties of the consensus matrix, while $V_3$ is a cross-term introduced to establish dissipativity and strict contraction.

First, we establish that $V$ is bounded by squared error norm, satisfying $\epsilon_4 \|\bz - \bz^0\|^2 \le V(\bz) \le \epsilon_3 \|\bz - \bz^0\|^2$. By interpreting the \texttt{q-PDGD} algorithm as a noisy discretization of a stable continuous-time dynamical system, we utilize the Lipschitz continuity of $\nabla V$. We show that the deterministic drift satisfies $\langle \nabla V(\bz), -G(\bz) \rangle \le -\frac{\epsilon_2}{\epsilon_3}V(\bz)$, which ensures exponential convergence in the noise-free case.

To account for the stochastic updates, we analyze the conditional expectation of the one-step difference $V_{k+1} - V_k$. The zero-mean property of the stochastic gradients and the quantization noise eliminates the linear noise terms. The remaining quadratic noise terms are bounded by the gradient variance $\sigma^2$ and the quantization parameter $\Delta^2$. This yields a recursive inequality of the form 
{\small
\begin{equation*}
    \mathbb{E}[V_{k+1}] \leq \rho \, \mathbb{E}[V_k] + O(\Delta^2 + \sigma^2),
\end{equation*}
}
\hspace{-8pt} 
where the contraction rate $\rho \in (0,1)$ depends on algorithm parameters. Unrolling this recursion and applying the quadratic bounds on $V$ yields the final error bound. The detailed derivation is provided in Appendix~\ref{sub:prf_thm1}.
\end{proof}

\begin{proof}[Proof Sketch of Theorem~\ref{thm:diminishing_step}]
The proof utilizes the same $V$ and the conditional expectation analysis from Theorem~\ref{thm:sq_pdgd_convergence}. By substituting the diminishing step size $h_k = \frac{\gamma}{k+1}$ into the established one-step bound, we obtain
{\small
\begin{equation*}
    \mathbb{E}[V_{k+1}] \leq \left(1 - \frac{\gamma \epsilon_2}{\epsilon_3(k+1)} + O\left(\frac{1}{k^2}\right)\right) \mathbb{E}[V_k] + \frac{\bar{C} \gamma^2}{(k+1)^2},
\end{equation*}
}
\hspace{-5pt} 
where $\bar{C}$ represents the constant noise terms. For sufficiently large $k$, the higher-order term $O(1/k^2)$ inside the parenthesis becomes negligible compared to the linear term. Letting $\mu = \frac{\gamma \epsilon_2}{2\epsilon_3}$, the inequality simplifies to
{\small
\begin{equation*}
    \mathbb{E}[V_{k+1}] \leq \left(1 - \frac{\mu}{k+1}\right) \mathbb{E}[V_k] + \frac{\bar{C} \gamma^2}{(k+1)^2}.
\end{equation*}
}
\hspace{-6pt}
The condition $\gamma > \frac{2\epsilon_3}{\epsilon_2}$ ensures that $\mu > 1$. Solving this scalar recursion yields the decay rate $\mathbb{E}[V_k] = O(1/k)$. Combining this with the quadratic bounds on $V$ completes the proof. See Appendix~\ref{sub:prf_thm2} for complete details.
\end{proof}

\begin{proof}[Proof Sketch of Theorem~\ref{thm:pl_diminishing_step}]
We employ a Lyapunov function $V$ similar to that in Theorem~\ref{thm:sq_pdgd_convergence}, comprising the consensus error, the auxiliary variable tracking error, and a cross-term, but we replace the distance to the optimal solution with the objective function gap $f(\bar{x}(k)) - f^*$. 

Analysis of the one-step dynamics establishes the recursion
{\small
\begin{equation*}
    \mathbb{E}[V_{k+1}] \le (1 - h c_{1})\mathbb{E}[V_k] + h \tilde{C}_A,
\end{equation*}
}
\hspace{-8pt} 
where $c_{1} > 0$ depends on the problem parameters and $\tilde{C}_A$ encapsulates the variance from stochastic gradients and quantization. The step size $h$ is chosen sufficiently small such that $(1 - h c_1) \in [0,1)$, ensuring a linear contraction of the deterministic energy. Unrolling this recursion yields
{\small
\begin{equation*}
    \mathbb{E}[V_k] \le (1 - h c_{1})^{k}V_0 + \frac{\tilde{C}_A}{c_{1}}.
\end{equation*}
}
\hspace{-6pt} 
As $k \to \infty$, the term depending on the initial condition vanishes geometrically, and the expected Lyapunov value converges to a neighborhood proportional to the noise level $\tilde{C}_A$. Specifically, we obtain $\limsup_{k \to \infty} \mathbb{E}[V_k] \le \frac{\tilde{C}_A}{c_1}$, which implies the stated asymptotic bound on the consensus and objective errors. The detailed proof is in Appendix~\ref{sub:prf_thm3}.
\end{proof}

\begin{remark}
For PL, due to the coupling between the gradient tracking error and the step size, the effective noise contribution $h\tilde{C}_A$ scales as $O(h)$, unlike the $O(h^2)$ scaling observed in the RSI case. Thus, for the weaker PL condition, a diminishing stepsize does not guarantee exact convergence.
\end{remark}

%% file: sections/experiments.tex
\section{Experimental Results}
\label{sec:exp}

We corroborate the theoretical predictions from Section~\ref{sec:conv} through controlled simulations under simultaneous stochastic gradients and finite-bit random quantization. We study two nonconvex objectives. All the communication graphs in this section are considered to be a {\em ring}.

\textbf{Example 1.} We consider heterogeneous local objectives obtained by shifting a quadratic-plus-periodic function across agents, defined by
\begin{equation}\label{eq:ex1_local_obj}
\begin{aligned}
f_i(x) &= (x-s_i)^2 + 3\sin^2(x-s_i); \, s_i \coloneqq i-1.
\end{aligned}
\end{equation}
This construction produces smooth but oscillatory gradients and does not rely on shared minimizers.

\textbf{Example 2.} We consider a piecewise-defined local objective that alternates between polynomial and curved (square-root) segments, creating multiple curvature regimes that stress stability under noise and quantization, defined by
\begin{equation}\label{eq:ex2_local_obj}
f_i(x)\!=\!
\begin{cases}
b_{1,i}(x+1)^2, & \! x\le -1,\\
b_{2,i}x^4, & \! -1 < x \le 0,\\
1-\sqrt{1-x^2}+b_{3,i}x^2, & \! 0 < x < a,\\
\sqrt{1-(x-c)^2}\!-\!c+\!1+\!b_{3,i}x^2, & \! a \le x < 1,\\
\begin{aligned}[t]
&\frac12\,(x-1+\kappa)^2 \;+\; b_{3,i}x^2\\
&\quad +\; \sqrt{2c-2} \;+\; \frac{5-5c}{4},
\end{aligned}
& \! x\ge 1.
\end{cases}
\end{equation}
Here $c \coloneqq \sqrt{2}, \, a \coloneqq \frac{c}{2}, \, \kappa \coloneqq \sqrt{\frac{c-1}{2}}$, and $b_{i,j}$ are randomly generated constants satisfying $\sum_{i=1}^n b_{i,1} > 0$, $\sum_{i=1}^n b_{i,2} = \sum_{i=1}^n b_{i,3} = 0$.

\textbf{Constant vs.\ diminishing step-sizes.}
To validate the theoretical guarantees, we evaluate the proposed \qpdgd{} algorithm under both constant and diminishing step-size regimes. Figure \ref{fig:convergence_comparison} shows the consensus error and the objective gap for the non-convex problems presented in Examples 1\&2 with $n=10$ agents. With a constant step-size, the method exhibits an initial linear convergence phase and then plateaus at a noise-dominated neighborhood, consistent with Theorem~\ref{thm:sq_pdgd_convergence}. In contrast, applying a diminishing step-size schedule $h_k = \mathcal{O}(1/k)$ enables the agents to overcome this noise floor, achieving asymptotic consensus and driving the objective error down at a sublinear rate of $\mathcal{O}(1/k)$, corroborating Theorem \ref{thm:diminishing_step}. 

\textbf{Network scaling.}
To validate the network-dependence derived in Remark~\ref{rem:AB_network}, we report how the total network error scales with the number of agents $n$ while keeping other parameters fixed. The observed trend in Figure~\ref{fig:network_scaling} matches the predicted dependence of the steady-state error on network size and spectral properties.

\textbf{Baseline comparison on least squares.}
We benchmark the practical efficiency of \qpdgd{} against standard distributed baselines, specifically Quantized Distributed Gradient Descent \texttt{q-DGD}~\citep{Yuanetal2013}, \texttt{CHOCO-SGD}~\citep{koloskova2019decentralized} and \texttt{DCSG}~\citep{dutta20241}. 
To establish a fair baseline against \texttt{q-DGD} and \texttt{DCSG}, which typically require stronger assumptions like strong convexity, we evaluate the algorithms on a standard distributed least squares problem. This fair comparison isolates and highlights the accelerated convergence of \qpdgd{}. In this problem, each agent's local cost is defined by
$$f_i(x_i) = \frac{1}{2}\|A_i x_i - b_i\|^2.$$
We simulate $n=10$ agents estimating a signal $x^* \in \mathbb{R}^3$ ($d=3$). The data matrices $A_i \in \mathbb{R}^{3 \times 3}$ and $x^*$ are drawn i.i.d. from standard normal distribution $\mathcal{N}(0,1)$, with measurements $b_i = A_i x^*$ ensuring a true minimum $f^* = 0$. Stochastic local gradients are computed as $g_i(x) = A_i^T(A_i x_i - b_i) + \zeta_i$, where $\zeta_i \sim \mathcal{N}(0,0.5)$. Communication is compressed using an 8-bit random uniform quantizer over $[-10,10]$. For this least-squares problem, Figure \ref{fig:quantized_vs_unquantized} showcases the objective gap trajectories under simultaneous stochastic gradient noise and random quantization. \qpdgd{} achieves the target optimality threshold of $0.01$ at $k=215$, significantly outperforming \texttt{q-DGD} ($k=745$), \texttt{CHOCO-SGD} ($k=948$) and \texttt{DCSG} ($k=969$). As further highlighted, \qpdgd{} exhibits remarkable robustness to finite-bit communication; its quantized trajectory closely mirrors its unquantized (exact communication) counterpart. We further stress-test \qpdgd{} against representative quantized decentralized baselines on two image-classification benchmarks where our assumptions are formally violated. As shown in Appendix~\ref{app:dl-experiments}, the consensus advantage predicted by Theorem~\ref{thm:sq_pdgd_convergence} continues to hold empirically. This confirms that the primal-dual tracking mechanism effectively absorbs and mitigates the consensus errors induced by random quantization without severely degrading the convergence rate. Additional plots demonstrating the dependency of final error on $\Delta^2$ and $\sigma^2$ from Theorem~\ref{thm:sq_pdgd_convergence} and performance comparison of \qpdgd{} with the aforementioned baselines to reach a target threshold are included in Appendix~\ref{app:plots}.

\begin{figure*}[t!]
    \centering
    \includegraphics[width=0.9\linewidth]{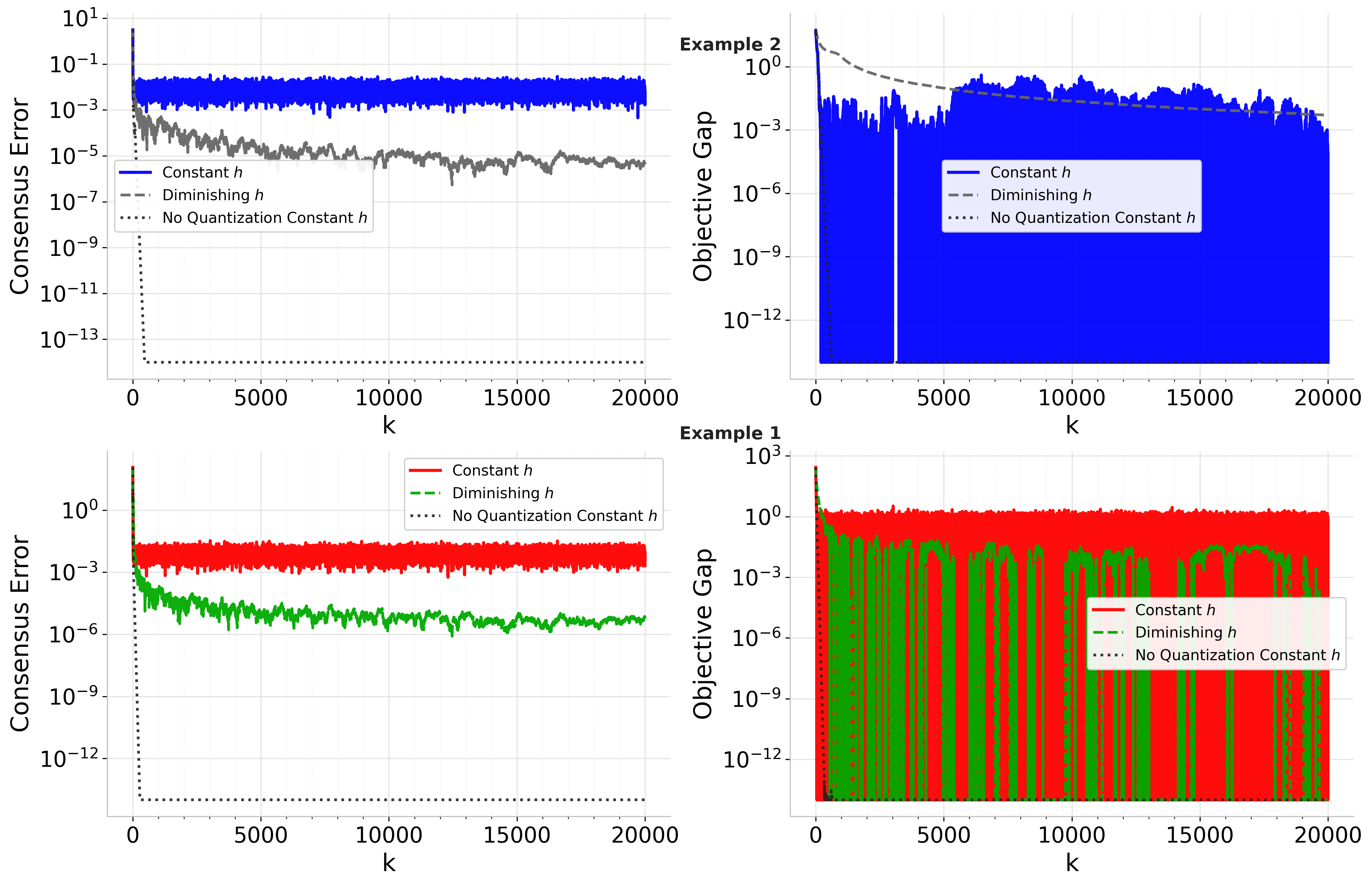}
    \captionof{figure}{\textbf{(Example 2, top) and (Example 1, bottom).} Evolution of the consensus error and objective gap for \qpdgd{} under constant and diminishing step-sizes over a ring graph with $n=10$ agents. Constant step-sizes converge linearly before plateauing at a neighborhood, while diminishing step-size achieves consensus and optimization.} 
    \label{fig:convergence_comparison}
    \begin{minipage}{0.48\textwidth}
        \centering
        \includegraphics[width=\linewidth]{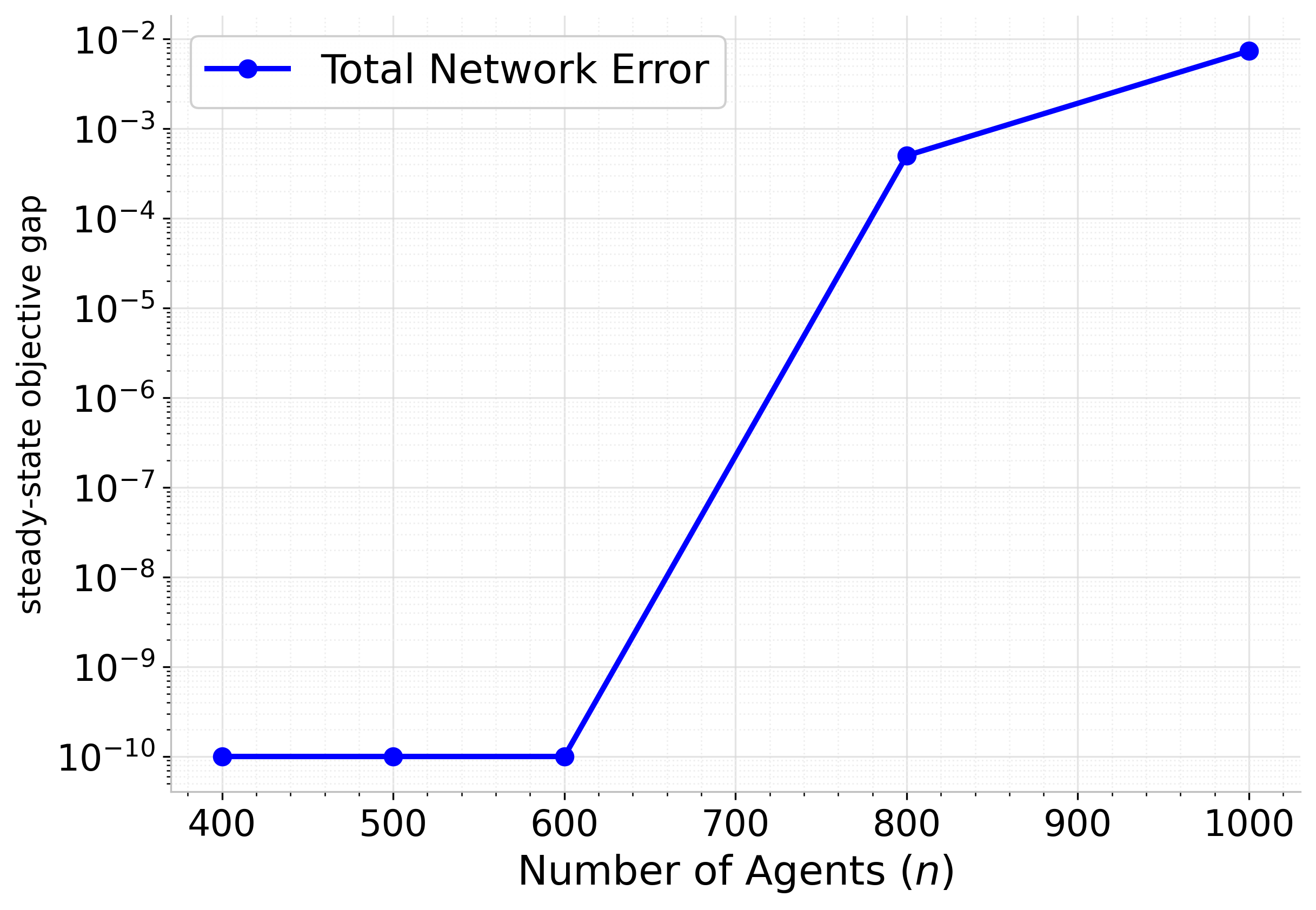}
        \captionof{figure}{Total network error scaling versus the number of agents $n$ for Example~2, validating the dependency predicted by Remark~\ref{rem:AB_network}.}
        \label{fig:network_scaling}
    \end{minipage}
    \hfill
    \begin{minipage}{0.48\textwidth}
        \centering
        \includegraphics[width=\linewidth]{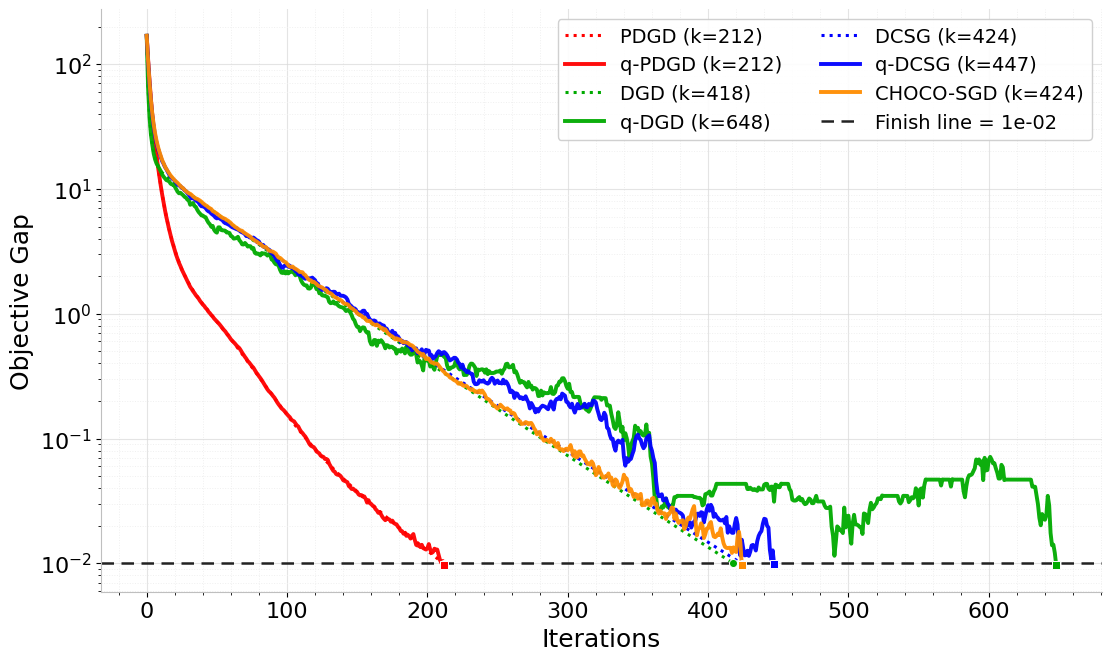}
        \captionof{figure}{Quantized (Q) versus unquantized (UQ) trajectories on distributed least squares under stochastic gradients and finite-bit random quantization. \qpdgd{} remains robust to compression. }
        \label{fig:quantized_vs_unquantized}
    \end{minipage}
    
\end{figure*}

%% file: sections/conclusion.tex
\section{Conclusion}
In summary, we developed \texttt{q-PDGD}, a stochastic primal–dual method that remains provably effective under simultaneous gradient noise and finite-bit, random (unbiased) quantized communication in fully decentralized networks. Under RSI, we established a clear linear-to-neighborhood guarantee for constant step-sizes with an explicit residual radius that cleanly decomposes the contributions of quantization distortion, stochastic variance, and graph connectivity, and we further showed that a diminishing step-size achieves the minimax-optimal $\bigO(1/k)$ rate without shared-minimizer assumptions. Under the even weaker PL geometry, we proved linear convergence to a noise-dominated neighborhood, again explicitly characterizing how network structure and compression affect the steady-state error. Experiments corroborate these predictions by exhibiting the expected plateau vs. $\bigO(1/k)$ behavior and by validating the linear dependence on ($\Delta^2$) and ($\sigma^2$) as well as the predicted scaling with network size and spectral properties. Collectively, these results position \texttt{q-PDGD} as a principled and practical tool for communication-constrained distributed learning, and they motivate future directions such as sharper constants, adaptive bit/step-size scheduling, and extensions to biased compressors and time-varying graphs.

%% file: sections/supplementary.tex
\onecolumn

\title{Quantized Stochastic Primal–Dual Methods for Distributed Optimization under Relaxed Global Geometry\\(Supplementary Material)}
\maketitle


\appendix

\input{sections/appendix}

%% file: sections/appendix.tex
\section{Detailed Proofs}
\label{sec:proofs}

\subsection{Matrix Notation and Spectral Properties}
\label{sub:mat_prop}

Recall $L$ denote the Laplacian matrix of the connected undirected graph $\mathcal{G}$. Let the {\em centering matrix} $K_n$ denote the orthogonal projection onto the subspace orthogonal to $\mathbf{1}_n$, defined as
\begin{equation}
K_n = \bI_n - \frac{1}{n}\mathbf{1}_n\mathbf{1}_n^\top.
\end{equation}
Both $L$ and $K_n$ are positive semi-definite. Since the graph is connected under Assumption~\ref{ass:connectivity}, the null spaces satisfy $\text{null}(L) = \text{null}(K_n) = \text{span}(\mathbf{1}_n)$. Furthermore, the matrices satisfy the commutativity and projection properties
\begin{equation}
K_n L = L K_n = L, \quad \text{and} \quad \rho(K_n) = 1.
\end{equation}

Let $0 = \lambda_1 < \lambda_2 \le \dots \le \lambda_n$ be the eigenvalues of $L$. We denote the spectral radius by $\rho(L) = \lambda_n$ and the algebraic connectivity (smallest non-zero eigenvalue) by $\rho_2(L) = \lambda_2$. The Laplacian satisfies the following bounds relative to $K_n$:
\begin{equation}
\rho_2(L)K_n \preceq L\preceq \rho(L)K_n.
\end{equation}

To analyze the weighted norm in the Lyapunov function, we introduce the spectral decomposition of $L$. There exists an orthogonal matrix $Q = [r, R] \in R^{n \times n}$, where $r = \frac{1}{\sqrt{n}}\mathbf{1}_n$ and $R \in R^{n \times (n-1)}$, such that~\citep{yi2020exponential}
\begin{align}
& L = [r \, R] 
\begin{bmatrix} 
0 & 0 \\ 
0 & \Lambda_1 
\end{bmatrix} 
\begin{bmatrix} 
r^\top \\ 
R^\top 
\end{bmatrix} 
= R \Lambda_1 R^\top, \label{eq:L_factor} \\
& R\Lambda_1^{-1}R^\top L = L R\Lambda_1^{-1}R^\top = K_n, \label{eq:Kn_def} \\
& \frac{1}{\rho(L)} K_n \preceq R\Lambda_1^{-1}R^\top \preceq \frac{1}{\rho_2(L)} K_n. \label{eq:Kn_bounds}
\end{align}
where $\Lambda_1 = \text{diag}([\lambda_2, \dots, \lambda_n]).$

\subsection{Proof of Lemma \ref{lem:rsi_augmented}}
\label{sub:prf_lem2}

Let $\bx^* = \mathcal{P}_{\mathcal{X}^*}(\bx)$. For any $\bx \in \mathbb{R}^{nd}$, we orthogonally decompose it as $\bx = \mathbf{a} + \mathbf{b}$, where $\mathbf{a} \in \mathbb{H} = \{\mathbf{1}_n \otimes y : y \in \mathbb{R}^d\}$ is the average component, and $\mathbf{b} \in \mathbb{H}^\perp$ is the disagreement component.
Specifically, $\mathbf{a} = \frac{1}{n}(\mathbf{1}_n\mathbf{1}_n^\top \otimes \bI_d)\bx$.

Since $\mathcal{X}^* \subseteq \mathbb{H}$, the projection onto the optimal set is determined solely by the average component, i.e., $\mathcal{P}_{\mathcal{X}^*}(\bx) = \mathcal{P}_{\mathcal{X}^*}(\mathbf{a}) = \bx^*$.
Consequently, the error norm decomposes as
\begin{equation}
\|\bx - \bx^*\|^2 = \|\mathbf{a} - \bx^*\|^2 + \|\mathbf{b}\|^2.
\end{equation}
\\
We expand the inner product term involving the gradients as
\begin{align}
&(\nablaF(\bx) - \nablaF(\bx^*))^\top (\bx - \bx^*) \notag \\
&= (\nablaF(\mathbf{a}) - \nablaF(\bx^*))^\top (\mathbf{a} - \bx^*) + (\nablaF(\bx) - \nablaF(\mathbf{a}))^\top (\mathbf{a} - \bx^*) + (\nablaF(\mathbf{a}) - \nablaF(\bx^*))^\top \mathbf{b} \notag \\
&+ (\nablaF(\bx) - \nablaF(\mathbf{a}))^\top \mathbf{b}. 
\end{align}
We bound these terms using the RSI Condition. Since $\mathbf{a}, \bx^* \in \mathbb{H}$, we can relate $F$ to the global cost $f$. From Assumption~\ref{ass:rsi},
   \begin{align} (\nablaF(\mathbf{a}) - \nablaF(\bx^*))^\top (\mathbf{a} - \bx^*) = 
   (\nabla f(a) - \nabla f(x^*))^\top (a - x^*) \ge \frac{\nu}{n}\|\mathbf{a} - \bx^*\|^2. \end{align}
From Lipschitz continuity in Assumption~\ref{ass:smooth},
\begin{align}
(\nablaF(\bx) - \nablaF(\mathbf{a}))^\top (\mathbf{a} - \bx^*) & \ge -L_f\|\bx - \mathbf{a}\|\|\mathbf{a} - \bx^*\| =-L_f\|\mathbf{b}\|\|\mathbf{a} - \bx^*\|, \\
(\nablaF(\bx) - \nablaF(\mathbf{a}))^\top \mathbf{b} & \ge -L_f\|\bx - \mathbf{a}\|\|\mathbf{b}\| = -L_f\|\mathbf{b}\|^2, \\
(\nablaF(\mathbf{a}) - \nablaF(\bx^*))^\top \mathbf{b} & \ge - L_f \|\mathbf{a} - \bx^*\| \|\mathbf{b}\|.
\end{align}
Combining the inequalities above,
\begin{align}
(\nablaF(\bx) - \nablaF(\bx^*))^\top (\bx - \bx^*) &\ge \frac{\nu}{n}\|\mathbf{a} - \bx^*\|^2 - 2 L_f\|\mathbf{b}\|\|\mathbf{a} - \bx^*\| 
- L_f\|\mathbf{b}\|^2. \label{eq:comb_lem2}
\end{align}
Using Young's inequality, we get
\[ -2L_f\|\mathbf{b}\|\|\mathbf{a} - \bx^*\| \ge -\frac{\nu}{2n}\|\mathbf{a} - \bx^*\|^2 - \frac{2nL_f^2}{\nu}\|\mathbf{b}\|^2. \]
Substituting above in~\eqref{eq:comb_lem2},
\begin{align}
(\nablaF(\bx) - \nablaF(\bx^*))^\top (\bx - \bx^*) &\ge \frac{\nu}{2n}\|\mathbf{a} - \bx^*\|^2 - \left( \frac{2nL_f^2}{\nu} + L_f \right)\|\mathbf{b}\|^2. \label{eq:grad_bound}
\end{align}
Now consider the Laplacian term. Since $\bx^\top \boldsymbol{L} \bx = \mathbf{b}^\top \boldsymbol{L} \mathbf{b}$ and $L \succeq \rho_2(L) K_n$ on the subspace $\mathbb{H}^\perp$,
\begin{equation}
\alpha \bx^\top \boldsymbol{L} \bx \ge \alpha \, \rho_2(L) \|\mathbf{b}\|^2. \label{eq:laplacian_bound}
\end{equation}
\\
Adding \eqref{eq:grad_bound} and \eqref{eq:laplacian_bound},
\begin{align}
& (\nablaF(\bx) - \nablaF(\bx^*))^\top (\bx - \bx^*) + \alpha \bx^\top \boldsymbol{L} \bx \notag \\
&\ge \frac{\nu}{2n}\|\mathbf{a} - \bx^*\|^2 + \left( \alpha \rho_2(L) - \frac{2nL_f^2}{\nu} - L_f \right)\|\mathbf{b}\|^2 \notag \\
&\ge \min \left\{ \frac{\nu}{2n}, \alpha \rho_2(L) - \frac{2nL_f^2 + \nu L_f}{\nu} \right\} (\|\mathbf{a} - \bx^*\|^2 + \|\mathbf{b}\|^2) \notag \\
&= \nu_1 \|\bx - \bx^*\|^2.
\end{align}
Given the condition on $\alpha$, the coefficient above is strictly positive, ensuring $\nu_1 > 0$.

\subsection{Proof of Theorem \ref{thm:sq_pdgd_convergence}}
\label{sub:prf_thm1}

Consider the following functions
\begin{align}
V_1(\bx, \bv) &= \frac{1}{2}\|\bx - \bx^0\|^2 + \frac{1}{2}\|\bv - \bv^0\|_{R\Lambda_1^{-1}R^\top \otimes \bI_d}^2, \label{eq:V1} \\
V_2(\bx, \bv) &= 2\epsilon_1 V_1(\bx, \bv) + \frac{\alpha}{2\beta}\|\bv - \bv^0\|^2, \label{eq:V2} \\
V_3(\bx, \bv) &= \bx^\top (K_n \otimes \bI_d)(\bv - \bv^0), \label{eq:V3}
\end{align}
where, $R$ and $\Lambda$ are defined in Section~\ref{sub:mat_prop}, $\bx^0 = \mathbf{1}_n \otimes x^0 = \Proj_{\calX}(\bx)$, and $\bv^0 = -\frac{1}{\beta}\nabla f(\bx^0)$.
Consider the Lyapunov function candidate
\begin{equation}
V(\bx, \bv) = V_2(\bx, \bv) + V_3(\bx, \bv). \label{eq:V}
\end{equation}

Since $\mathbf{1}_n^\top L = \mathbf{0}_n^\top$, the $v_i$-update in~\eqref{eq:update_law} implies that the sum of the auxiliary variables is invariant with respect to $k$, i.e., $\sum_{i=1}^n v_i(k+1) = \sum_{i=1}^n v_i(k)$. So, given the initialization $\sum_{i=1}^n v_i(0) = \mathbf{0}_d$, we have $\sum_{i=1}^n v_i(k) = \mathbf{0}_p$ for all $k \ge 0$. This implies that the vector $\bv$ lies entirely in the range space of the centering matrix $(K_n \otimes \bI_d)$. Furthermore, since $\bv^0 = -\frac{1}{\beta}\nabla f(\bx^0)$ and $\sum_{i=1}^n \nabla f_i(x^0) = \mathbf{0}_p$, we have $(K_n \otimes \bI_d)(\bv - \bv^0) = \bv - \bv^0$.

Recall from~\eqref{eq:Kn_bounds} that the matrix $R\Lambda_1^{-1}R^\top \otimes \bI_d$ satisfies the following bounds:
\begin{equation}
\frac{1}{\rho(L)} (K_n \otimes \bI_d) \preceq R\Lambda_1^{-1}R^\top \otimes \bI_d \preceq \frac{1}{\rho_2(L)} (K_n \otimes \bI_d).
\end{equation}
Using above, we bound the second term in $V_1$ as
\begin{align*}
\|\bv -\bv^0\|_{R\Lambda_1^{-1}R^\top \otimes \bI_d}^2
\le \frac{1}{\rho_2(L)} (\bv - \bv^0)^\top (K_n \otimes \bI_d) (\bv - \bv^0) = \frac{1}{\rho_2(L)}\|\bv - \bv^0\|^2.
\end{align*}

Next, we bound the cross-term $V_3$ using Young's inequality. Noting that $\bx^0 = \mathbf{1}_n \otimes x^0$, we have $(\bx^0)^\top (K_n \otimes \bI_d) = \mathbf{0}_{nd}$. Thus, $V_3$ can be rewritten as
\[
V_3(\bx, \bv) = (\bx - \bx^0)^\top (K_n \otimes \bI_d) (\bv - \bv^0).
\]
Using Young's inequality and $\|K_n \otimes \bI_d\| = 1$, we obtain
\begin{align*}
-\frac{1}{2}(\|\bx - \bx^0\|^2 + \|\bv - \bv^0\|^2) \le V_3(\bx, \bv)
\le \frac{1}{2}(\|\bx - \bx^0\|^2 + \|\bv - \bv^0\|^2).
\end{align*}

Substituting these bounds into the definition of $V(\bx, \bv)$ in~\eqref{eq:V},
\begin{align}
V &\le 2\epsilon_1 \left( \frac{1}{2}\|\bx - \bx^0\|^2 + \frac{1}{2\rho_2(L)}\|\bv - \bv^0\|^2 \right)  + \frac{\alpha}{2\beta}\|\bv - \bv^0\|^2 + \frac{1}{2}(\|\bx - \bx^0\|^2 + \|\bv - \bv^0\|^2) \notag\\
&\le \epsilon_3 (\|\bx - \bx^0\|^2 + \|\bv - \bv^0\|^2), \label{eq:Vbd}
\end{align}
where $\epsilon_3 = \max \left( \epsilon_1 + \frac{1}{2}, \frac{\epsilon_1}{\rho_2(L)} + \frac{\alpha}{2\beta} + \frac{1}{2} \right)$.
Similarly, for the lower bound of $V$, we can obtain
\begin{align}
V &\ge 2\epsilon_1 \left( \frac{1}{2}\|\bx - \bx^0\|^2 \right) + \frac{\alpha}{2\beta}\|\bv - \bv^0\|^2 - \frac{1}{2}(\|\bx - \bx^0\|^2 + \|\bv - \bv^0\|^2) \notag\\
&\ge \epsilon_4 (\|\bx - \bx^0\|^2 + \|\bv - \bv^0\|^2),
\end{align}
where $\epsilon_4 = \min \left( \epsilon_1 - \frac{1}{2}, \frac{\alpha}{2\beta} - \frac{1}{2} \right)$. By choosing,
$\epsilon_1 > \frac{1}{2}$ and $\frac{\alpha}{\beta} > 1$, we ensure $\epsilon_4 > 0$.
Thus, 
\begin{align}
\label{eq:Vbounds}
\epsilon_4 (\|\bv - \bv^0\|^2 + \|\bx - \bx^0\|^2) \le V(\bz)
\le \epsilon_3 (\|\bv - \bv^0\|^2 + \|\bx - \bx^0\|^2).
\end{align}

Recall the state vector $\bz(k) = [\bx(k)^\top, \bv(k)^\top]^\top$.
The Lyapunov function $V(\bz)$ is differentiable and its gradient is $\nabla V(\bz) = [\nabla V(\bz)_1^\top, \nabla V(\bz)_2^\top]^\top$, where
\begin{align*}
[\nabla V(\bz)]_1 &= 2\epsilon_1(\bx - \bx^0) + (K_n \otimes \bI_d)(\bv - \bv^0), \\
[\nabla V(\bz)]_2 &= (K_n \otimes \bI_d)\bx + \left(2\epsilon_1(R\Lambda_1^{-1}R^\top \otimes \bI_d) + \frac{\alpha}{\beta} \otimes \bI_{nd}\right)(\bv - \bv^0).
\end{align*}
Also, $\nabla V(\bz)$ is Lipschitz continuous with constant $\eta = \sqrt{2} \max\left\{\frac{2\epsilon_1}{\rho_2(L)} + \alpha + 1, 4\epsilon_1 + 1\right\} > 0$.
Under Assumption~\ref{ass:smooth} using the descent lemma for smooth functions, and substituting from~\eqref{eq:z_dyn},
\begin{align}
& V(\bz(k+1)) - V(\bz(k))
\le \langle\nabla V(\bz(k)), \bz(k+1)-\bz(k)\rangle + \frac{\eta}{2}\|\bz(k+1)-\bz(k)\|^2 \notag\\
&= -h\langle\nabla V(\bz(k)), G(\bz(k))\rangle - h\langle\nabla V(\bz(k)), \xi(k) + \frac{\eta h^2}{2} (\|G(\bz(k))\|^2 + \|\xi(k)\|^2 \notag\\
&+ 2\langle G(\bz(k)), \xi(k) \rangle). \label{eq:V_diff_1}
\end{align}\\
To derive a bound for the deterministic term $-\langle\nabla V, F\rangle$, we observe that $-\langle \nabla V(\bz), G(\bz) \rangle$ corresponds to the time-derivative $\dot{V}$ of the same Lyapunov function along the trajectories of only the deterministic component of the continuous-time counterpart of~\eqref{eq:z_dyn}, defined by
\begin{subequations}
\label{eq:continuous_dynamics}
\begin{align}
\dot{\bx}(t) &= -\alpha \boldsymbol{L}\bx(t) - \beta \bv(t) - \nablaF(\bx(t)), \\
\dot{\bv}(t) &= \beta \boldsymbol{L}\bx(t).
\end{align}
\end{subequations}
We have, $\mathbf{1}^\top \bv(t) = \mathbf{1}^\top \bv(0) = \mathbf{0}_{nd}$. Considering the projection matrix  $P = K_n \otimes \bI_d$, we have
\[
P(\bv - \bv^0) = \bv - \bv^0.
\]
Recalling $V_1(\bx, \bv) = \frac{1}{2}\|\bx - \bx^0\|^2 + \frac{1}{2}\|\bv - \bv^0\|_{R\Lambda_1^{-1}R^\top \otimes \bI_d}^2$, the derivative is
\begin{align}
\dot{V}_1 &= (\bx - \bx^0)^\top \dot{\bx} + (\bv - \bv^0)^\top (R\Lambda_1^{-1}R^\top \otimes \bI_d) \dot{\bv} \notag \\
&= (\bx - \bx^0)^\top (-\alpha \boldsymbol{L}\bx - \beta \bv - \nablaF(\bx)) + (\bv - \bv^0)^\top (R\Lambda_1^{-1}R^\top \otimes \bI_d) (\beta \boldsymbol{L}\bx).
\end{align}
Using $(R\Lambda_1^{-1}R^\top \otimes \bI_d)\boldsymbol{L} = K_n \otimes \bI_d = P$, $\boldsymbol{L}\bx^0 = \mathbf{0}$, $P(\bv - \bv^0) = \bv - \bv^0$, and Lemma \ref{lem:rsi_augmented},
\begin{align}
\dot{V}_1 &= -\alpha \bx^\top \boldsymbol{L}\bx - (\bx - \bx^0)^\top (\nablaF(\bx) - \nablaF(\bx^0)) \le -\nu_1 \|\bx - \bx^0\|^2. \label{eq:V1_dot_explicit}
\end{align}
Differentiating $V_2(\bz) = 2\epsilon_1 V_1(\bz) + \frac{\alpha}{2\beta}\|\bv - \bv^0\|^2$,
\begin{align}
\dot{V}_2 = 2\epsilon_1 \dot{V}_1 + \frac{\alpha}{\beta}(\bv - \bv^0)^\top \dot{\bv} &= 2\epsilon_1 \dot{V}_1 + \frac{\alpha}{\beta}(\bv - \bv^0)^\top (\beta \boldsymbol{L}\bx) \notag \\
& \le -2\epsilon_1 \nu_1 \|\bx - \bx^0\|^2 + \alpha (\bv - \bv^0)^\top \boldsymbol{L}\bx. \label{eq:V2_dot_explicit}
\end{align}
Differentiating $V_3(\bz) = (\bx - \bx^0)^\top (K_n \otimes \bI_d) (\bv - \bv^0)$,
\begin{align}
\dot{V}_3
&= (\bx - \bx^0)^\top (K_n \otimes \bI_d) \dot{\bv} + \dot{\bx}^\top (K_n \otimes \bI_d) (\bv - \bv^0) \notag \\
&\le \beta \rho(L)\|\bx - \bx^0\|^2 - \alpha (\bv - \bv^0)^\top \boldsymbol{L}\bx - \beta \|\bv - \bv^0\|^2 + \frac{L_f^2}{2\beta}\|\bx - \bx^0\|^2 + \frac{\beta}{2}\|\bv - \bv^0\|^2 \notag \\
&= -\alpha (\bv - \bv^0)^\top \boldsymbol{L}\bx - \frac{\beta}{2}\|\bv - \bv^0\|^2 + \left(\beta \rho(L) + \frac{L_f^2}{2\beta}\right)\|\bx - \bx^0\|^2.
\label{eq:V3_dot_explicit}
\end{align}
Adding \eqref{eq:V2_dot_explicit} and \eqref{eq:V3_dot_explicit},
\begin{align}
\dot{V} &= \dot{V}_2 + \dot{V}_3 \le \left( -2\epsilon_1 \nu_1 + \beta \rho(L) + \frac{L_f^2}{2\beta} \right) \|\bx - \bx^0\|^2 - \frac{\beta}{2}\|\bv - \bv^0\|^2.
\end{align}
By choosing $\epsilon_1 \ge \frac{1}{\nu_1}\left(\beta \rho(L) + \frac{L_f^2}{2\beta}\right)$, we have
\begin{equation}
\dot{V} \le -\epsilon_1 \nu_1 \|\bx - \bx^0\|^2 - \frac{\beta}{2}\|\bv - \bv^0\|^2.
\end{equation}
Define $\epsilon_2 = \min\left\{
\epsilon_1 \nu_1, \frac{\beta}{2}
\right\}.$ From the upper bound in~\eqref{eq:Vbounds} and above, we conclude $\dot{V} \le -\frac{\epsilon_2}{\epsilon_3}V$.
We have obtained that the first term in the RHS of~\eqref{eq:V_diff_1} is bounded by
\begin{equation}
-F^\top(\bz(k)) \nabla V(\bz(k)) = \dot{V}(\bz(k)) \le -\frac{\epsilon_2}{\epsilon_3} V(\bz(k)). \label{eq:drift_bound_thm}
\end{equation}

Next, we upper bound the other deterministic term $\|F\|^2$ in~\eqref{eq:V_diff_1}.
Using $\boldsymbol{L}\bx^0 = \mathbf{0}_{nd}$ and $\bv^0 = -\frac{1}{\beta}\nablaF(\bx^0)$, we rewrite
\[
G(\bz) = \begin{bmatrix} \alpha \boldsymbol{L}(\bx - \bx^0) + \beta(\bv - \bv^0) + \nablaF(\bx) - \nablaF(\bx^0) \\ -\beta \boldsymbol{L}(\bx - \bx^0) \end{bmatrix}.
\]
Using the triangle inequality, Lipschitz continuity of $\nablaF$ from Assumption~\ref{ass:smooth}, and Section~\ref{sub:mat_prop},
\begin{align*}
\|G(\bz)\|^2 &\le 3\|\alpha \boldsymbol{L}(\bx - \bx^0)\|^2 + 3\|\beta(\bv - \bv^0)\|^2 + 3\|\nablaF(\bx) - \nablaF(\bx^0)\|^2 + \|\beta \boldsymbol{L}(\bx - \bx^0)\|^2 \\
&\le (\beta^2 \rho^2(L) + 3\alpha^2 \rho^2(L) + 3L_f^2) \|\bx - \bx^0\|^2 + 3\beta^2 \|\bv - \bv^0\|^2 \\
&\le \epsilon_5 (\|\bx - \bx^0\|^2 + \|\bv - \bv^0\|^2),
\end{align*}
where $\epsilon_5 = \max\{ 3\alpha^2 \rho^2(L) + \beta^2 \rho^2(L) + 3L_f^2, 3\beta^2 \}$.
Using the lower bound from~\eqref{eq:Vbounds}, we obtain
\begin{equation}
\|G(\bz(k))\|^2 \le \frac{\epsilon_5}{\epsilon_4} V(\bz(k)). \label{eq:grad_norm_thm}
\end{equation}
Upon substituting from above in~\eqref{eq:V_diff_1},
\begin{align*}
&V(\bz(k+1)) \\
&\leq \!V(\bz(k))
\!-\!\frac{h \epsilon_2}{\epsilon_3}\!V(\bz(k))
\!+\!\frac{\eta h^2 \epsilon_5}{2 \epsilon_1}\!V(\bz(k))
\!-\!h\!\langle\nabla V(\bz(k)), \!\xi(k)\rangle
\!+\!\frac{\eta h^2}{2}\!(\|\xi(k)\|^2 \!+\! 2\!\langle G(\bz(k)),\!\xi(k) \rangle)\\
&= (1-\frac{h(2\epsilon_2\epsilon_4 - h\eta\epsilon_3\epsilon_5)}{2\epsilon_3\epsilon_4}) V(\bz(k)) -h\langle\nabla V(\bz(k)), \xi(k)\rangle
+\frac{\eta h^2}{2}(\|\xi(k)\|^2 + 2\langle G(\bz(k)),\xi(k) \rangle).
\end{align*}

Taking conditional expectation with respect to $\calF_k$, and using the fact that the random quantization model and stochastic gradients are unbiased (Assumption~\ref{ass:stoch_grad} and~\eqref{eq:noise_bound}), we have
\begin{align*}
\E[V(\bz(k+1))\mid\calF_k] 
&\leq (1-\frac{h(2\epsilon_2\epsilon_4 - h\eta\epsilon_3\epsilon_5)}{2\epsilon_3\epsilon_4}) V(z(k))-h\E[\langle\nabla V(\bz(k)), \xi(k)\rangle\mid\calF_k] \\
&+ \eta h^2\E[\langle G(\bz(k)),\xi(k) \rangle\mid\calF_k]
+\frac{\eta h^2}{2}\E[\|\xi(k)\|^2\mid\calF_k]\\
\implies \E[V(\bz(k+1))] 
&\leq (1-\frac{h(2\epsilon_2\epsilon_4 - h\eta\epsilon_3\epsilon_5)}{2\epsilon_3\epsilon_4}) \E[V(\bz(k))] +\frac{\eta h^2}{2}\E[\|\xi(k)\|^2],
\end{align*}
as $\E[\xi(k) \mid \calF_k] = \mathbf{0}_{2nd}$.
Moreover, from Assumption~\ref{ass:stoch_grad} and~\eqref{eq:noise_bound}, it follows that
\[
\mathbb{E}[\|\xi(k)\|^2]
= \E[\|\alpha \mathcal{L}\varepsilon(k)+\zeta(k)\|^2] + \E[\|\beta \mathcal{L}\varepsilon(k)\|^2]\\
\le (2\alpha^2 + \beta^2)\rho^2(L)\E[\|\varepsilon(k)\|^2]
+ 2\E[\|\zeta(k)\|^2],
\]
so that
\[
\mathbb{E}[\|\xi(k)\|^2] 
\le (2\alpha^2 + \beta^2)\rho^2(L)\frac{nd\Delta^2}{4} + 2n\sigma^2 
\]
Combining these bounds yields
\begin{align}
    \E[V(\bz(k+1))] & \leq (1-\frac{h(2\epsilon_2\epsilon_4 - h\eta\epsilon_3\epsilon_5)}{2\epsilon_3\epsilon_4}) \E[V(\bz(k))] \!+\!\frac{\eta h^2}{2}\! \left((2\alpha^2 + \beta^2)\rho^2(L)\frac{nd\Delta^2}{4} \!+\! 2n\sigma^2 \right) \notag \\
    & = (1-\frac{h(2\epsilon_2\epsilon_4 - h\eta\epsilon_3\epsilon_5)}{2\epsilon_3\epsilon_4}) \E[V(\bz(k))] +\frac{\eta h^2}{2} (C_1\Delta^2 + C_2\sigma^2). \label{eq:vk_recur_rsi}
\end{align}
Since $h$ satisfies $0 < h < \min \left\{\frac{2 \epsilon_2 \epsilon_4}{\eta \epsilon_3 \epsilon_5},\, \frac{2\epsilon_3}{\epsilon_2}\right\}$, we have $0 < (1-\frac{h(2\epsilon_2\epsilon_4 - h\eta\epsilon_3\epsilon_5)}{2\epsilon_3\epsilon_4}) < 1$. Then, unrolling the above recursion from $k$ to $0$,
\begin{align*}
    \E[V(\bz(k))] & \leq (1-\frac{h(2\epsilon_2\epsilon_4 - h\eta\epsilon_3\epsilon_5)}{2\epsilon_3\epsilon_4})^k\E[V(\bz(0))] + \frac{\epsilon_3\epsilon_4\eta h}{(2\epsilon_2\epsilon_4 - h\eta\epsilon_3\epsilon_5)}(C_1\Delta^2 + C_2\sigma^2) \\
    &= (1-\frac{h(2\epsilon_2\epsilon_4 - h\eta\epsilon_3\epsilon_5)}{2\epsilon_3\epsilon_4})^k\E[V(\bz(0))] + K_1\Delta^2 + K_2\sigma^2.
\end{align*}

Then, from~\eqref{eq:Vbounds} the definition of $\bx^0$,
\[
\E[\|\bx(k)-\Proj_{\calX}(\bx(k))\|^2] \leq \frac{\E[V(\bz(k))]}{\epsilon_4}
\leq (1-\frac{h(2\epsilon_2\epsilon_4 - h\eta\epsilon_3\epsilon_5)}{2\epsilon_3\epsilon_4})^k\frac{\E[V(\bz(0))]}{\epsilon_4} + A\Delta^2 + B\sigma^2.
\]
Thus, 
\[
\E[\|\bx(k)-\Proj_{\calX}(\bx(k))\|]
\leq (1-\frac{h(2\epsilon_2\epsilon_4 - h\eta\epsilon_3\epsilon_5)}{4\epsilon_3\epsilon_4})^k\sqrt\frac{\E[V(\bz(0))]}{\epsilon_4} + \sqrt{A\Delta^2 + B\sigma^2}.
\]
The proof is complete.


\subsection{Proof of Theorem \ref{thm:diminishing_step}}
\label{sub:prf_thm2}

Consider the same Lyapunov function
\[
V(\bz)=V_2(\bx,\bv)+V_3(\bx,\bv),
\]
defined in the proof of Theorem~\ref{thm:sq_pdgd_convergence} in~\eqref{eq:V1}-\eqref{eq:V}.
Under the Assumptions~\ref{ass:diff}-\ref{ass:rsi}, \ref{ass:singleton}-\ref{ass:stoch_grad}, and the conditions on parameters $\alpha, \beta$ from Theorem~\ref{thm:sq_pdgd_convergence}, we follow its proof, we obtain the recursive step~\eqref{eq:vk_recur_rsi} as
\[
\E[V(\bz(k+1))] \leq (1-\frac{h_k(2\epsilon_2\epsilon_4 - h_k\eta\epsilon_3\epsilon_5)}{2\epsilon_3\epsilon_4}) \E[V(\bz(k))] +\frac{\eta h_k^2}{2} (C_1\Delta^2 + C_2\sigma^2).
\]
\textit{Note}: The stepsize condition $0 < h < \min \left\{\frac{2 \epsilon_2 \epsilon_4}{\eta \epsilon_3 \epsilon_5}, \frac{2\epsilon_3}{\epsilon_2}\right\}$ was not necessary for obtaining~\eqref{eq:vk_recur_rsi}.
Unlike the Proof of Theorem~\ref{thm:sq_pdgd_convergence}, we will not show contraction from $k=0$; since it is enough to show contraction for large enough $k$ to prove Theorem~\ref{thm:diminishing_step}. So, we do not need the $0 < h < \min \left\{\frac{2 \epsilon_2 \epsilon_4}{\eta \epsilon_3 \epsilon_5}, \frac{2\epsilon_3}{\epsilon_2}\right\}$ condition. Instead, we will leverage a diminishing stepsize for this purpose. 

Let $\bar{C} := \frac{\eta}{2} (C_1\Delta^2 + C_2\sigma^2)$.
Substituting \( h_k = \frac{\gamma}{k+1} \),
\[
\E[V(\bz(k+1))]
\leq
\left(
1-
\frac{\frac{\gamma}{k+1}
\left(
2\epsilon_2\epsilon_4
-
\frac{\gamma}{k+1}\eta\epsilon_3\epsilon_5
\right)}
{2\epsilon_3\epsilon_4}
\right)
\E[V(\bz(k))]
+
\frac{\gamma^2 \bar{C}}{(k+1)^2}.
\]
Expanding the coefficient of $\E[V(\bz(k))]$,
\[
1-
\frac{\gamma}{k+1}
\cdot
\frac{2\epsilon_2\epsilon_4}{2\epsilon_3\epsilon_4}
+
\frac{\gamma}{k+1}
\cdot
\frac{\frac{\gamma}{k+1}\eta\epsilon_3\epsilon_5}
{2\epsilon_3\epsilon_4}
=
1-
\frac{\gamma\epsilon_2}{\epsilon_3(k+1)}
+
\frac{\gamma^2 \eta \epsilon_5}{2\epsilon_4 (k+1)^2}.
\]
Thus,
\[
\E[V(\bz(k+1))]
\leq
\left(
1-
\frac{\gamma\epsilon_2}{\epsilon_3(k+1)}
+
\frac{\gamma^2 \eta \epsilon_5}{2\epsilon_4 (k+1)^2}
\right)
\E[V(\bz(k))]
+
\frac{\gamma^2 \bar{C}}{(k+1)^2}
\]
For sufficiently large \(k\), since \( \frac{1}{(k+1)^2} = o\!\left(\frac{1}{k+1}\right) \), there exists \(K\) such that for all \(k \geq K\),
\[
\frac{\gamma^2 \eta \epsilon_5}{2\epsilon_4 (k+1)^2}
\leq
\frac{1}{2}
\cdot
\frac{\gamma\epsilon_2}{\epsilon_3(k+1)}.
\]
Hence,
\[
1-
\frac{\gamma\epsilon_2}{\epsilon_3(k+1)}
+
\frac{\gamma^2 \eta \epsilon_5}{2\epsilon_4 (k+1)^2}
\leq
1-
\frac{1}{2}
\cdot
\frac{\gamma\epsilon_2}{\epsilon_3(k+1)}.
\]
Defining $\mu = \frac{\gamma\epsilon_2}{2\epsilon_3}$, we obtain for all sufficiently large $k \geq K$,
\[
\E[V(\bz(k+1))]
\leq
\left(1-\frac{\mu}{k+1}\right)
\E[V(\bz(k))]
+
\frac{\gamma^2 \bar{C}}{(k+1)^2}.
\]
We require $\gamma$ such that $\mu > 1$. Next, we unroll the above from $k$ to $K$:
\[
\E[V(\bz(k))] 
\leq 
\left( \prod_{i=K+1}^{k} \left(1 - \frac{\mu}{i}\right) \right) 
\E[V(\bz(K))]
+
\sum_{j=K+1}^{k}
\left(
\prod_{m=j+1}^{k}
\left(1 - \frac{\mu}{m}\right)
\right)
\frac{\gamma^2 \bar{C}}{j^2}.
\]
We bound the coefficients as follows.
Using $1 - x \leq e^{-x}$,
\[
\prod_{i=K+1}^{k} \left(1 - \frac{\mu}{i}\right)
\leq
\exp\left(-\mu \sum_{i=K+1}^{k} \frac{1}{i}\right).
\]
Since
\[
\sum_{i=K+1}^{k} \frac{1}{i}
\geq
\ln(k+1) - \ln(K+1),
\]
we obtain
\[
\prod_{i=K+1}^{k} \left(1 - \frac{\mu}{i}\right)
\leq
\left(\frac{K+1}{k+1}\right)^{\mu}.
\]
Similarly,
\[
\prod_{m=j+1}^{k} \left(1 - \frac{\mu}{m}\right)
\leq
\left(\frac{j+1}{k+1}\right)^{\mu}.
\]
Substituting the above inequalities for the coefficients,
\begin{align*}
    \E[V(\bz(k))]
    & \leq
    \left(\frac{K+1}{k+1}\right)^{\mu}
    \E[V(\bz(K))]
    +
    \sum_{j=K+1}^{k}
    \left(\frac{j+1}{k+1}\right)^{\mu}
    \frac{\gamma^2 \bar{C}}{j^2} \\
    & = \frac{(K+1)^{\mu}\E[V(\bz(K))]}{(k+1)^{\mu}}
+
\frac{\gamma^2 \bar{C}}{(k+1)^{\mu}}
\sum_{j=K+1}^{k}
\frac{(j+1)^{\mu}}{j^2}.
\end{align*}
Since $(j+1)^{\mu} \leq 2^{\mu} j^{\mu}$ for $j \geq 1$, we have the summation term
\[
\sum_{j=K+1}^{k} \frac{(j+1)^{\mu}}{j^2}
\leq
2^{\mu} \sum_{j=K+1}^{k} j^{\mu-2}.
\]
For $\mu > 1$,
\[
\sum_{j=K+1}^{k} j^{\mu-2}
\leq
\int_{K+1}^{k+1} x^{\mu-2} dx
=
\frac{(k+1)^{\mu-1} - (K+1)^{\mu-1}}{\mu-1}.
\]
Substituting the summation term with this bound,
\begin{align*}
    \E[V(\bz(k))]
    & \leq
    \frac{(K+1)^{\mu}\E[V(\bz(K))]}{(k+1)^{\mu}}
    +
    \frac{\gamma^2 \bar{C} \, 2^{\mu}}{(k+1)^{\mu}}
    \frac{(k+1)^{\mu-1} - (K+1)^{\mu-1}}{\mu-1} \\
    & \leq
    \frac{(K+1)^{\mu}\E[V(\bz(K))]}{(k+1)^{\mu}}
    +
    \frac{\gamma^2 \bar{C} \, 2^{\mu}}{(k+1)^{\mu}}
    \frac{(k+1)^{\mu-1}}{\mu-1} \\
    & =
    \frac{(K+1)^{\mu}\E[V(\bz(K))]}{(k+1)^{\mu}}
    +
    \frac{\gamma^2 \bar{C} \, 2^{\mu}}{(k+1) (\mu-1)}.
\end{align*}
Then, from~\eqref{eq:Vbounds} the definition of $\bx^0$,
\[
\E[\|\bx(k)-\Proj_{\calX}(\bx(k))\|^2] \leq \frac{\E[V(\bz(k))]}{\epsilon_4}
\leq 
\frac{(K+1)^{\mu}\E[V(\bz(K))]}{(k+1)^{\mu} \, \epsilon_4}
+
\frac{\gamma^2 \bar{C} \, 2^{\mu}}{(k+1) (\mu-1) \epsilon_4}.
\]
The RHS is $\bigo\left(\frac{1}{(k+1)^{\mu}}\right) + \bigo\left(\frac{1}{k+1}\right)$. Given $\mu > 1$, the dominant term is of order $O\left(\frac{1}{k+1}\right)$. So, the proof is complete.

\input{sections/pl_proof}
\input{sections/complexity}
\input{sections/ML}

%% file: sections/pl_proof.tex
\subsection{Proof of Theorem \ref{thm:pl_diminishing_step}}
\label{sub:prf_thm3}

We denote the following, which will be used throughout this proof:
\[
\boldsymbol{K} := K_n \otimes \bI_d,\qquad
\boldsymbol{H} := \frac{1}{n}(\boldsymbol{1}_n \boldsymbol{1}_n^\top \otimes \bI_d),\qquad
\boldsymbol{Q} := R\Lambda^{-1}R^\top \otimes \bI_d,
\]
\[
\bar{v}(k) := \frac{1}{n}(\boldsymbol{1}_n^\top \otimes \bI_d)\boldsymbol{v}(k), 
\qquad
\bar{\zeta}(k) := \frac{1}{n}(\boldsymbol{1}_n^\top \otimes \bI_d)\boldsymbol{\zeta}(k), \qquad
\bar{\boldsymbol{\zeta}}(k) := \boldsymbol{1}_n \otimes \bar{\zeta}(k)
\]
\[
\boldsymbol{g}(k) := \nabla F(\boldsymbol{x}(k)),\qquad
\bar{\boldsymbol{g}}(k) := \boldsymbol{H} \boldsymbol{g}(k)
\]
\[
\boldsymbol{g}^0(k) := \nabla F(\bar{\boldsymbol{x}}(k)),\qquad
\bar{\boldsymbol{g}}^0(k) := \boldsymbol{H} \boldsymbol{g}^0(k)
= \frac{1}{n}(\boldsymbol{1}_n \otimes \nabla f(\bar{\boldsymbol{x}}(k))).
\]
We again follow a Lyapunov-based technique as formulated below.
Define the Lyapunov function
\[
V(k) := \sum_{i=1}^4 V_{i}(k),
\]
where
\begin{align*}
V_{1}(k) &= \frac12 \|\boldsymbol{x}(k)\|_{\boldsymbol{K}}^2,\\
V_{2}(k) &= \frac12\left\|\boldsymbol{v}(k) + \frac{1}{\beta} \boldsymbol{g}^0(k) \right\|_{\boldsymbol{Q}+\frac{\alpha}{\beta}\boldsymbol{K}}^2,\\
V_{3}(k) &= \boldsymbol{x}(k)^\top \boldsymbol{K}\left(\boldsymbol{v}(k) + \frac{1}{\beta} \boldsymbol{g}^0(k)\right),\\
V_{4}(k) &= f(\bar{x}(k)) - f^* = F(\bar{\boldsymbol{x}}(k)) - f^*.
\end{align*}
As shown in the paragraph following~\eqref{eq:V} under Assumption~\ref{ass:connectivity}, from~\eqref{eq:update_law} we have
\[
\bar{v}(k+1) = \bar{v}(k).
\]
So, given the initialization $\sum_{i=1}^n v_i(0) = \mathbf{0}_d$, we have $\bar{v}(k) = 0_{d}$ for all $k \ge 0$. Upon substituting into~\eqref{eq:update_law},
\begin{align}
\bar{\boldsymbol{x}}(k+1)
=
\bar{\boldsymbol{x}}(k)
-
h \bar{\boldsymbol{g}}(k)
-
h \bar{\boldsymbol{\zeta}}(k).
\label{eq:avg_x}
\end{align}

Next, we establish necessary inequalities which will be utilized later in the Lyapunov analysis. 
Since $\nablaF$ is Lipschitz-continuous with constant $L_f > 0$ and $\rho(\boldsymbol{H}) = 1$, from above we have
\begin{align}
\E[\|\boldsymbol{g}^0(k+1) - \boldsymbol{g}^0(k)\|^2] \leq L_f^2\E[\|\boldsymbol{x}(k+1)-\boldsymbol{x}(k)\|^2] & = L_f^2\E[\|-h \bar{\boldsymbol{g}}(k)
-h \bar{\boldsymbol{\zeta}}(k)\|^2]
\notag\\ 
& \leq
2L_f^2h^2n\sigma^2 + L_f^2h^2\E[\|\bar{\boldsymbol{g}}(k)\|^2],
\label{eq:grad1}
\end{align}
where the last inequality follows from unbiased stochastic gradients under Assumption~\ref{ass:stoch_grad}. By definition and Lipschitz continuity from Assumption~\ref{ass:smooth}, the following holds:
\begin{align}
\E[\|\boldsymbol{g}^0(k)-\boldsymbol{g}(k)\|^2] & \leq L_f^2\E[\|\bar{\boldsymbol{x}}(k) - \boldsymbol{x}(k)\|^2] = L_f^2
\E[\|\boldsymbol{x}(k)\|^2_{\boldsymbol{K}}] = L_f^2\|\boldsymbol{x}(k)\|^2_{\boldsymbol{K}},
\label{eq:grad2} \\
\E[\|\bar{\boldsymbol{g}}^0(k)-\bar{\boldsymbol{g}}(k)\|^2] &= \E[\|\boldsymbol{H}(\boldsymbol{g}^0(k)-\boldsymbol{g}(k))\|^2] \leq \E[\|\boldsymbol{g}^0(k)-\boldsymbol{g}(k)\|^2] \leq L_f^2\|\boldsymbol{x}(k)\|^2_{\boldsymbol{K}}.
\label{eq:grad3}
\end{align}
From Polyak--Łojasiewicz inequality under Assumption~\ref{ass:PL}, the global objective satisfies
\begin{align}
\|\nabla f(\bar{\boldsymbol{x}}(k))\|^2 \ge 2\nu \bigl(f(\bar{\boldsymbol{x}}(k))-f^*\bigr).
\end{align}
Equivalently,
\begin{align}
\|\bar{\boldsymbol{g}}^0(k)\|^2 = \frac{1}{n}\|\nabla f(\bar{\boldsymbol{x}}(k))\|^2
\ge \frac{2\nu}{n}\bigl(f(\bar{\boldsymbol{x}}(k))-f^*\bigr).
\label{eq:PL}
\end{align}

Next, we upper bound the first component $\E[V_{1}(k+1)]$ in the Lyapunov function $V(k+1)$.
\begin{align}
&\E[V_{1}(k+1)] \notag\\
&= \frac{1}{2}\E[\|\boldsymbol{x}(k+1)\|_{\boldsymbol{K}}^2] \notag\\
&= \frac{1}{2}\E[\|\boldsymbol{x}(k) - h(\alpha \boldsymbol{L} \boldsymbol{x}(k) + \alpha\boldsymbol{L}\boldsymbol{\varepsilon}(k)+ \beta \boldsymbol{v}(k) + \boldsymbol{g}(k) + \boldsymbol{\zeta}(k))\|_{\boldsymbol{K}}^2] \notag\\
&\leq \frac{1}{2}\|\boldsymbol{x}(k) - h(\alpha \boldsymbol{L} \boldsymbol{x}(k) + \beta \boldsymbol{v}(k) + \boldsymbol{g}(k))\|_{\boldsymbol{K}}^2 +h^2 \alpha^2\rho^2(\boldsymbol{L})\E[\|\boldsymbol{\varepsilon}(k)\|^2] + h^2\E[\|\boldsymbol{\zeta}(k)\|^2]\notag\\
&\leq \frac{1}{2}\|\boldsymbol{x}(k)\|_{\boldsymbol{K}}^2
- h\alpha \|\boldsymbol{x}(k)\|_{\boldsymbol{L}}^2
+ \frac{h^2\alpha^2}{2}\|\boldsymbol{x}(k)\|_{\boldsymbol{L}^2}^2 \notag\\
&\quad
+ \frac{h^2\beta^2}{2}\left\|\boldsymbol{v}(k) + \frac{1}{\beta}\boldsymbol{g}(k)\right\|_{\boldsymbol{K}}^2
- h\beta \boldsymbol{x}(k)^\top (\boldsymbol{I}_{nd}-h\alpha \boldsymbol{L})\boldsymbol{K}
\left(\boldsymbol{v}(k) + \frac{1}{\beta}\boldsymbol{g}(k)\right) \notag\\
&+h^2 \alpha^2\rho^2(\boldsymbol{L})\frac{nd\Delta^2}{4}
+ h^2n\sigma^2 \notag\\
&= \frac{1}{2}\|\boldsymbol{x}(k)\|_{\boldsymbol{K}}^2
- \|\boldsymbol{x}(k)\|_{h\alpha \boldsymbol{L} - \frac{h^2\alpha^2}{2}\boldsymbol{L}^2}^2
+ \frac{h^2\beta^2}{2}
\left\|\boldsymbol{v}(k) + \frac{1}{\beta}\boldsymbol{g}^0(k)
+ \frac{1}{\beta}\boldsymbol{g}(k)
- \frac{1}{\beta}\boldsymbol{g}^0(k)\right\|_{\boldsymbol{K}}^2 \notag\\
&\quad
- h\beta \boldsymbol{x}(k)^\top (\boldsymbol{I}_{nd}-h\alpha \boldsymbol{L})\boldsymbol{K}
\left(\boldsymbol{v}(k) + \frac{1}{\beta}\boldsymbol{g}^0(k)
+ \frac{1}{\beta}\boldsymbol{g}(k)
- \frac{1}{\beta}\boldsymbol{g}^0(k)\right) +h^2 \alpha^2\rho^2(\boldsymbol{L})\frac{nd\Delta^2}{4} \notag\\
&+ h^2n\sigma^2 \notag\\
&\le \frac{1}{2}\|\boldsymbol{x}(k)\|_{\boldsymbol{K}}^2
- \|\boldsymbol{x}(k)\|_{h\alpha \boldsymbol{L} - \frac{h^2\alpha^2}{2}\boldsymbol{L}^2}^2
+ h^2\beta^2\left\|\boldsymbol{v}(k) + \frac{1}{\beta}\boldsymbol{g}^0(k)\right\|_{\boldsymbol{K}}^2
+ h^2\|\boldsymbol{g}(k) - \boldsymbol{g}^0(k)\|^2 \notag\\
&\quad
- h\beta \boldsymbol{x}(k)^\top \boldsymbol{K}\left(\boldsymbol{v}(k) + \frac{1}{\beta}\boldsymbol{g}^0(k)\right)
+ \frac{h}{2}\|\boldsymbol{x}(k)\|_{\boldsymbol{K}}^2
+ \frac{h}{2}\|\boldsymbol{g}(k) - \boldsymbol{g}^0(k)\|^2 \notag\\
&\quad
+ \frac{h^2\alpha^2}{2}\|\boldsymbol{x}(k)\|_{\boldsymbol{L}^2}^2
+ \frac{h^2\beta^2}{2}
\left\|\boldsymbol{v}(k) + \frac{1}{\beta}\boldsymbol{g}^0(k)\right\|_{\boldsymbol{K}}^2
+ \frac{h^2}{2}\|\boldsymbol{g}(k) - \boldsymbol{g}^0(k)\|^2 +h^2 \alpha^2\rho^2(\boldsymbol{L})\frac{nd\Delta^2}{4} \notag\\
&+ h^2n\sigma^2 \notag\\
&= \frac{1}{2}\|\boldsymbol{x}(k)\|_{\boldsymbol{K}}^2
- \|\boldsymbol{x}(k)\|_{h\alpha \boldsymbol{L} - \frac{h}{2}\boldsymbol{K} - \frac{3h^2\alpha^2}{2}\boldsymbol{L}^2}^2
+ \frac{h}{2}(1+3h)\|\boldsymbol{g}(k) - \boldsymbol{g}^0(k)\|^2 \notag\\
&\quad
- h\beta \boldsymbol{x}(k)^\top \boldsymbol{K}\left(\boldsymbol{v}(k) + \frac{1}{\beta}\boldsymbol{g}^0(k)\right)
+ \left\|\boldsymbol{v}(k) + \frac{1}{\beta}\boldsymbol{g}^0(k)\right\|_{\frac{3h^2\beta^2}{2}\boldsymbol{K}}^2 +h^2 \alpha^2\rho^2(\boldsymbol{L})\frac{nd\Delta^2}{4} + h^2n\sigma^2 \notag\\
&\le \frac{1}{2}\|\boldsymbol{x}(k)\|_{\boldsymbol{K}}^2
- \|\boldsymbol{x}(k)\|_{h\alpha \boldsymbol{L} - \frac{h}{2}\boldsymbol{K} - \frac{3h^2\alpha^2}{2}\boldsymbol{L}^2
- \frac{h}{2}(1+3h)L_f^2\boldsymbol{K}}^2 \notag\\
&\quad
- h\beta \boldsymbol{x}(k)^\top \boldsymbol{K}\left(\boldsymbol{v}(k) + \frac{1}{\beta}\boldsymbol{g}^0(k)\right)
+ \left\|\boldsymbol{v}(k) + \frac{1}{\beta}\boldsymbol{g}^0(k)\right\|_{\frac{3h^2\beta^2}{2}\boldsymbol{K}}^2 +h^2 \alpha^2\rho^2(\boldsymbol{L})\frac{nd\Delta^2}{4} + h^2n\sigma^2,
\label{eq:V1_bound}
\end{align}
where, the first inequality follows from Cauchy-Schwarz inequality and $\rho(K) = 1$; the second equality is due to~\eqref{eq:update_law}; the third equality follows from the spectral properties in Appendix~\ref{sub:mat_prop}; and the last inequality follows from~\eqref{eq:grad2}.

Similarly, we bound the second component $\E[V_{2}(k+1)]$ next.
\begin{align}
&\E\!\left[V_{2}(k+1)\right] \notag\\
&= \frac{1}{2}\E\!\left[\|\boldsymbol{v}(k+1) + \frac{1}{\beta}\boldsymbol{g}^0(k+1)\|_{\boldsymbol{Q}+\frac{\alpha}{\beta}\boldsymbol{K}}^2 \right] \notag\\
&= \frac{1}{2}\E\!\left[\|\boldsymbol{v}(k) + \frac{1}{\beta}\boldsymbol{g}^0(k) + h\beta \boldsymbol{L}\boldsymbol{x}(k) + h\beta \boldsymbol{L}\boldsymbol{\varepsilon}(k) + \frac{1}{\beta}(\boldsymbol{g}^0(k+1) - \boldsymbol{g}^0(k))\|_{\boldsymbol{Q}+\frac{\alpha}{\beta}\boldsymbol{K}}^2 \right] \notag\\
&\le \frac{1}{2}\|\boldsymbol{v}(k) + \frac{1}{\beta}\boldsymbol{g}^0(k)\|_{\boldsymbol{Q}+\frac{\alpha}{\beta}\boldsymbol{K}}^2 + h\boldsymbol{x}(k)^\top (\beta \boldsymbol{K} + \alpha \boldsymbol{L})(\boldsymbol{v}(k) + \frac{1}{\beta}\boldsymbol{g}^0(k)) \notag\\
&\quad + \|\boldsymbol{x}(k)\|_{\frac{h^2\beta}{2}(\beta \boldsymbol{L}+\alpha \boldsymbol{L}^2)}^2 + \frac{1}{2\beta^2}\E[\|\boldsymbol{g}^0(k+1) - \boldsymbol{g}^0(k)\|_{\boldsymbol{Q}+\frac{\alpha}{\beta}\boldsymbol{K}}^2] \notag\\
&\quad + \frac{1}{\beta}(\boldsymbol{v}(k) + \frac{1}{\beta}\boldsymbol{g}^0(k) + h\beta \boldsymbol{L}\boldsymbol{x}(k))^\top \left(\boldsymbol{Q} + \frac{\alpha}{\beta}\boldsymbol{K}\right)\E[\boldsymbol{g}^0(k+1) - \boldsymbol{g}^0(k)] \notag\\
&\quad + \frac{h^2\beta^2}{2}\E[\|\boldsymbol{L}\boldsymbol{\varepsilon}(k)\|_{\boldsymbol{Q}+\frac{\alpha}{\beta}\boldsymbol{K}}^2]  + \E[\left( h \beta \boldsymbol{L} \boldsymbol{\varepsilon(k)}\right)^\top\Big(\boldsymbol{Q}+\frac{\alpha}{\beta}\boldsymbol{K}\Big)\frac{1}{\beta}\left( \boldsymbol{g}^0(k+1)-\boldsymbol{g}(k)^0\right)]
\notag\\
&\le \frac{1}{2}\|\boldsymbol{v}(k) + \frac{1}{\beta}\boldsymbol{g}^0(k)\|_{\boldsymbol{Q}+\frac{\alpha}{\beta}\boldsymbol{K}}^2 + h\boldsymbol{x}(k)^\top (\beta \boldsymbol{K} + \alpha \boldsymbol{L})(\boldsymbol{v}(k) + \frac{1}{\beta}\boldsymbol{g}^0(k)) \notag\\
&\quad + \|\boldsymbol{x}(k)\|_{\frac{h^2\beta}{2}(\beta \boldsymbol{L}+\alpha \boldsymbol{L}^2)}^2 + \|\boldsymbol{v}(k) + \frac{1}{\beta}\boldsymbol{g}^0(k)\|_{\frac{h}{2\beta}(\boldsymbol{Q}+\frac{\alpha}{\beta}\boldsymbol{K})}^2 \notag\\
&\quad + \frac{1}{2\beta^2}\E[\|\boldsymbol{g}^0(k+1) - \boldsymbol{g}^0(k)\|_{\boldsymbol{Q}+\frac{\alpha}{\beta}\boldsymbol{K}}^2] + \frac{1}{2h\beta}\E[\|\boldsymbol{g}^0(k+1) - \boldsymbol{g}^0(k)\|_{\boldsymbol{Q}+\frac{\alpha}{\beta}\boldsymbol{K}}^2] \notag\\
&\quad + \frac{h^2\beta^2}{2}\|\boldsymbol{L}\boldsymbol{x}(k)\|_{\boldsymbol{Q}+\frac{\alpha}{\beta}\boldsymbol{K}}^2 + \frac{1}{2\beta^2}\E[\|\boldsymbol{g}^0(k+1) - \boldsymbol{g}^0(k)\|_{\boldsymbol{Q}+\frac{\alpha}{\beta}\boldsymbol{K}}^2] \notag\\
&\quad + h^2\beta^2\rho\left(\boldsymbol{Q}+\frac{\alpha}{\beta}\boldsymbol{K}\right)\rho^2(\boldsymbol{L})\frac{nd\Delta^2}{4} + \frac{1}{2\beta^2}\E[\|\boldsymbol{g}^0(k+1) - \boldsymbol{g}^0(k)\|_{\boldsymbol{Q}+\frac{\alpha}{\beta}\boldsymbol{K}}^2]\notag\\
&\le \frac{1}{2}\|\boldsymbol{v}(k) + \frac{1}{\beta}\boldsymbol{g}^0(k)\|_{\boldsymbol{Q}+\frac{\alpha}{\beta}\boldsymbol{K}}^2 + h\boldsymbol{x}(k)^\top (\beta \boldsymbol{K} + \alpha \boldsymbol{L})(\boldsymbol{v}(k) + \frac{1}{\beta}\boldsymbol{g}^0(k)) \notag\\
&\quad + \|\boldsymbol{x}(k)\|_{h^2\beta(\beta \boldsymbol{L}+\alpha \boldsymbol{L}^2)}^2 + \|\boldsymbol{v}(k) + \frac{1}{\beta}\boldsymbol{g}^0(k)\|_{\frac{h}{2\beta}(\boldsymbol{Q}+\frac{\alpha}{\beta}\boldsymbol{K})}^2 \notag\\
&\quad + \left( \frac{3}{2\beta^2} + \frac{1}{2h\beta} \right)\E[\|\boldsymbol{g}^0(k+1) - \boldsymbol{g}^0(k)\|_{\boldsymbol{Q}+\frac{\alpha}{\beta}\boldsymbol{K}}^2]  + h^2\beta^2\left( \frac{1}{\rho_2(L)} + \frac{\alpha}{\beta} \right)\rho^2(\boldsymbol{L})\frac{nd\Delta^2}{4} \notag\\
&\le \frac{1}{2}\|\boldsymbol{v}(k) + \frac{1}{\beta}\boldsymbol{g}^0(k)\|_{\boldsymbol{Q}+\frac{\alpha}{\beta}\boldsymbol{K}}^2 + h\boldsymbol{x}(k)^\top (\beta \boldsymbol{K} + \alpha \boldsymbol{L})(\boldsymbol{v}(k) + \frac{1}{\beta}\boldsymbol{g}^0(k)) \notag\\
&\quad + \|\boldsymbol{x}(k)\|_{h^2\beta(\beta \boldsymbol{L}+\alpha \boldsymbol{L}^2)}^2 + \|\boldsymbol{v}(k) + \frac{1}{\beta}\boldsymbol{g}^0(k)\|_{\frac{h}{2\beta}(\boldsymbol{Q}+\frac{\alpha}{\beta}\boldsymbol{K})}^2 \notag\\
&\quad + h\left( \frac{3h}{2\beta^2} + \frac{1}{2\beta} \right)\left( \frac{1}{\rho_2(L)} + \frac{\alpha}{\beta} \right) \left( L_f^2\|\boldsymbol{\bar{g}}(k)\|^2 + 2L_f^2 n\sigma^2 \right) \notag\\
&+ h^2\beta^2\left( \frac{1}{\rho_2(L)} + \frac{\alpha}{\beta} \right)\rho^2(\boldsymbol{L})\frac{nd\Delta^2}{4},
\label{eq:V2_bound}
\end{align}
where, the first inequality follows from Cauchy-Schwarz inequality; the second equality is due to~\eqref{eq:update_law}; the third equality follows from the spectral properties in Appendix~\ref{sub:mat_prop}; the third inequality holds since $\rho(Q + \frac{\alpha}{\beta}K) \le \rho(Q) + \frac{\alpha}{\beta} \rho(K)$, the spectral properties in Appendix~\ref{sub:mat_prop}, $ \rho(K)= 1$; and the last inequality holds follows from~\eqref{eq:grad1}.

We bound the third component $\E[V_{3}(k+1)]$ next.
\begin{align}
&\E[V_{3}(k+1)] \notag\\
&= \E[\boldsymbol{x}(k+1)^\top \boldsymbol{K} (\boldsymbol{v}(k+1) + \frac{1}{\beta}\boldsymbol{g}^0(k+1))] \notag\\
&= \E[(\boldsymbol{x}(k) - h(\alpha \boldsymbol{L}\boldsymbol{x}(k) + \beta \boldsymbol{v}(k) + \boldsymbol{g}(k) + \boldsymbol{g}^0(k) - \boldsymbol{g}^0(k) + \boldsymbol{\zeta}(k) + \alpha\boldsymbol{L}\boldsymbol{\varepsilon}(k)))^\top \boldsymbol{K} \notag\\
&\quad \times (\boldsymbol{v}(k) + \frac{1}{\beta}\boldsymbol{g}^0(k) + h\beta \boldsymbol{L}\boldsymbol{x}(k) + h\beta\boldsymbol{L}\boldsymbol{\varepsilon}(k) + \frac{1}{\beta}(\boldsymbol{g}^0(k+1) - \boldsymbol{g}^0(k)))] \notag\\
&= \boldsymbol{x}(k)^\top (\boldsymbol{K} - h(\alpha + h\beta^2)\boldsymbol{L})(\boldsymbol{v}(k) + \frac{1}{\beta}\boldsymbol{g}^0(k)) \notag\\
&\quad + \|\boldsymbol{x}(k)\|_{h\beta(\boldsymbol{L}-h\alpha \boldsymbol{L}^2)}^2 + \frac{1}{\beta}\boldsymbol{x}(k)^\top (\boldsymbol{K} - h\alpha \boldsymbol{L})\E[\boldsymbol{g}^0(k+1) - \boldsymbol{g}^0(k)] \notag\\
&\quad - h\beta\|\boldsymbol{v}(k) + \frac{1}{\beta}\boldsymbol{g}^0(k)\|_{\boldsymbol{K}}^2 - h(\boldsymbol{v}(k) + \frac{1}{\beta}\boldsymbol{g}^0(k))^\top \boldsymbol{K} \E[\boldsymbol{g}^0(k+1) - \boldsymbol{g}^0(k)] \notag\\
&\quad - h(\boldsymbol{g}(k) - \boldsymbol{g}^0(k))^\top \boldsymbol{K} (\boldsymbol{v}(k) + \frac{1}{\beta}\boldsymbol{g}^0(k) + h\beta \boldsymbol{L}\boldsymbol{x}(k) + \frac{1}{\beta}\E[\boldsymbol{g}^0(k+1) - \boldsymbol{g}^0(k)]) \notag\\
&\quad + \frac{1}{\beta}\E[(\boldsymbol{g}^0(k+1) - \boldsymbol{g}^0(k))^\top \boldsymbol{K} \boldsymbol{x}(k+1)] - h^2\alpha\beta\E[\boldsymbol{\varepsilon}(k)^\top \boldsymbol{L} \boldsymbol{K} \boldsymbol{L} \boldsymbol{\varepsilon}(k)] \notag\\
&- h^2 \beta \E[\boldsymbol{\zeta}(k)^\top \boldsymbol{K} \boldsymbol{L} \boldsymbol{\varepsilon}(k)] \notag\\
&\le \boldsymbol{x}(k)^\top (\boldsymbol{K} - h\alpha \boldsymbol{L})(\boldsymbol{v}(k) + \frac{1}{\beta}\boldsymbol{g}^0(k)) + \frac{h^2\beta^2}{2}\|\boldsymbol{L}\boldsymbol{x}(k)\|^2 \notag\\
&\quad + \frac{h^2\beta^2}{2}\|\boldsymbol{v}(k) + \frac{1}{\beta}\boldsymbol{g}^0(k)\|_{\boldsymbol{K}}^2 + \|\boldsymbol{x}(k)\|_{h\beta(\boldsymbol{L}-h\alpha \boldsymbol{L}^2)}^2 \notag\\
&\quad + \frac{h}{2}\|\boldsymbol{x}(k)\|_{\boldsymbol{K}}^2 + \frac{1}{2h\beta^2}\E[\|\boldsymbol{g}^0(k+1) - \boldsymbol{g}^0(k)\|^2] + \frac{h^2\alpha^2}{2}\|\boldsymbol{L}\boldsymbol{x}(k)\|^2 \notag\\
&\quad + \frac{1}{2\beta^2}\E[\|\boldsymbol{g}^0(k+1) - \boldsymbol{g}^0(k)\|^2] - h\beta\|\boldsymbol{v}(k) + \frac{1}{\beta}\boldsymbol{g}^0(k)\|_{\boldsymbol{K}}^2 \notag\\
&\quad + \frac{h^2}{2}\|\boldsymbol{v}(k) + \frac{1}{\beta}\boldsymbol{g}^0(k)\|_{\boldsymbol{K}}^2 + \frac{1}{2}\E[\|\boldsymbol{g}^0(k+1) - \boldsymbol{g}^0(k)\|^2] \notag\\
&\quad + \frac{h}{2}\|\boldsymbol{g}(k) - \boldsymbol{g}^0(k)\|^2 + \frac{h}{2}\|\boldsymbol{v}(k) + \frac{1}{\beta}\boldsymbol{g}^0(k)\|_{\boldsymbol{K}}^2 + \frac{h^2}{2}\|\boldsymbol{g}(k) - \boldsymbol{g}^0(k)\|^2 \notag\\
&\quad + \frac{h^2\beta^2}{2}\|\boldsymbol{L}\boldsymbol{x}(k)\|^2 + \frac{h^2}{2}\|\boldsymbol{g}(k) - \boldsymbol{g}^0(k)\|^2 + \frac{1}{2\beta^2}\E[\|\boldsymbol{g}^0(k+1) - \boldsymbol{g}^0(k)\|^2] \notag\\
&+ \frac{h^2 \beta}{2} (n\sigma^2 + \rho^2(\boldsymbol{L})\frac{nd\Delta^2}{4}) \notag\\
&= \boldsymbol{x}(k)^\top (\boldsymbol{K} - h\alpha \boldsymbol{L})(\boldsymbol{v}(k) + \frac{1}{\beta}\boldsymbol{g}^0(k)) + \frac{h}{2}(1 + 2h)\|\boldsymbol{g}(k) - \boldsymbol{g}^0(k)\|^2 \notag\\
&\quad + \|\boldsymbol{x}(k)\|_{h(\beta \boldsymbol{L}+\frac{1}{2}\boldsymbol{K})+h^2(\frac{\alpha^2}{2}-\alpha\beta+\beta^2)\boldsymbol{L}^2}^2 \notag\\
&\quad + (\frac{1}{2h\beta^2} + \frac{1}{\beta^2} + \frac{1}{2})\E[\|\boldsymbol{g}^0(k+1) - \boldsymbol{g}^0(k)\|^2] 
- \|\boldsymbol{v}(k) + \frac{1}{\beta}\boldsymbol{g}^0(k)\|_{h(\beta-\frac{1}{2}-\frac{h}{2}-\frac{h\beta^2}{2})\boldsymbol{K}}^2 \notag \\
&+ \frac{h^2 \beta}{2} (n\sigma^2 + \rho^2(\boldsymbol{L})\frac{nd\Delta^2}{4}) \notag\\
&\le \boldsymbol{x}(k)^\top \boldsymbol{K} (\boldsymbol{v}(k) + \frac{1}{\beta}\boldsymbol{g}^0(k)) - h\alpha \boldsymbol{x}(k)^\top \boldsymbol{L} (\boldsymbol{v}(k) + \frac{1}{\beta}\boldsymbol{g}^0(k)) \notag\\
&\quad + \|\boldsymbol{x}(k)\|_{h(\beta \boldsymbol{L}+\frac{1}{2}\boldsymbol{K})+h^2(\frac{\alpha^2}{2}-\alpha\beta+\beta^2)\boldsymbol{L}^2+\frac{h}{2}(1+2h)L_f^2\boldsymbol{K}}^2 \notag\\
&\quad + h(\frac{1}{2\beta^2} + \frac{h}{\beta^2} + \frac{h}{2}) \left( L_f^2\|\boldsymbol{\bar{g}}(k)\|^2 + 2L_f^2 n\sigma^2 \right) 
- \|\boldsymbol{v}(k) + \frac{1}{\beta}\boldsymbol{g}^0(k)\|_{h(\beta-\frac{1}{2}-\frac{h}{2}-\frac{h\beta^2}{2})\boldsymbol{K}}^2 \notag \\
& + \frac{h^2 \beta}{2} (n\sigma^2 + \rho^2(\boldsymbol{L})\frac{nd\Delta^2}{4}),
\label{eq:V3_bound}
\end{align}
where, the first inequality follows from the Cauchy-Schwarz inequality, the spectral properties in Appendix~\ref{sub:mat_prop}, and $\rho(K)=1$; the second equality is due to~\eqref{eq:update_law}; the third equality follows from the spectral properties in Appendix~\ref{sub:mat_prop}; and the last inequality follows from~\eqref{eq:grad1}-\eqref{eq:grad2}.\\

Finally, we bound the fourth component $\E[V_{4}(k+1)]$ next.
\begin{align}
&\E[V_{4}(k+1)] \notag\\
&= \E[F(\bar{\boldsymbol{x}}(k+1)) - f^*] \notag\\
&= \E[F(\bar{\boldsymbol{x}}(k)) - f^* + F(\bar{\boldsymbol{x}}(k+1)) - F(\bar{\boldsymbol{x}}(k))] \notag\\
&\le F(\bar{\boldsymbol{x}}(k)) - f^* - h \E[\bar{\boldsymbol{g}}(k)^\top \boldsymbol{g}^0(k)] + \frac{h^2 L_f}{2}\E[\|\bar{\boldsymbol{g}}(k)\|^2] 
+ \frac{L_f}{2}\E\left[\left\| - h \bar{\boldsymbol{\zeta}}(k)\right\|^2\right] \notag\\
&= f(\bar{\boldsymbol{x}}(k)) - f^* - \frac{h}{2}\bar{\boldsymbol{g}}(k)^\top (\bar{\boldsymbol{g}}(k) + \bar{\boldsymbol{g}}^0(k) - \bar{\boldsymbol{g}}(k)) \notag\\
&\quad - \frac{h}{2}(\bar{\boldsymbol{g}}(k) - \bar{\boldsymbol{g}}^0(k) + \bar{\boldsymbol{g}}^0(k))^\top \bar{\boldsymbol{g}}^0(k) + \frac{h^2 L_f}{2}\|\bar{\boldsymbol{g}}(k)\|^2 
+ \frac{L_f}{2}\left( 2h^2 n \sigma^2 \right) \notag\\
&\le f(\bar{\boldsymbol{x}}(k)) - f^* - \frac{h}{4}\|\bar{\boldsymbol{g}}(k)\|^2 + \frac{h}{4}\|\bar{\boldsymbol{g}}^0(k) - \bar{\boldsymbol{g}}(k)\|^2 - \frac{h}{4}\|\bar{\boldsymbol{g}}^0(k)\|^2 \notag\\
&\quad + \frac{h}{4}\|\bar{\boldsymbol{g}}^0(k) - \bar{\boldsymbol{g}}(k)\|^2 + \frac{h^2 L_f}{2}\|\bar{\boldsymbol{g}}(k)\|^2 
+ L_f h^2 n \sigma^2 \notag\\
&= f(\bar{\boldsymbol{x}}(k)) - f^* - \frac{h}{4}(1 - 2h L_f)\|\bar{\boldsymbol{g}}(k)\|^2 + \frac{h}{2}\|\bar{\boldsymbol{g}}^0(k) - \bar{\boldsymbol{g}}(k)\|^2 
- \frac{h}{4}\|\bar{\boldsymbol{g}}^0(k)\|^2 + L_f h^2 n \sigma^2 \notag\\
&\le f(\bar{\boldsymbol{x}}(k)) - f^* - \frac{h}{4}(1 - 2h L_f)\|\bar{\boldsymbol{g}}(k)\|^2 + \|\boldsymbol{x}(k)\|_{\frac{h}{2}L_f^2 \boldsymbol{K}}^2 
- \frac{h\nu}{2n}(f(\bar{\boldsymbol{x}}(k)) - f^*) + L_f h^2 n \sigma^2,
\label{eq:V4_bound}
\end{align}
where the first inequality holds follows from smoothness of $F$,~\eqref{ass:smooth} and~\eqref{eq:avg_x}; the second inequality is due to Cauchy-Schwarz inequality;
the third equality holds since $\boldsymbol{\bar{g}(k)^\top} \boldsymbol{g}^0(k)= \boldsymbol{g}(k)^\top \boldsymbol{H} \boldsymbol{g}^0(k) = \boldsymbol{g(}k)^\top \boldsymbol{H} \boldsymbol{H} \boldsymbol{g}^0(k) = \boldsymbol{\bar{g}(k)^\top} \bar{\boldsymbol{g}}^0(k)$;
and the last inequality follows from~\eqref{eq:grad3} and~\eqref{eq:PL}.

Combining the above components of $\E[V(k+1)]$ in~\eqref{eq:V1_bound}-\eqref{eq:V4_bound}, we have
\begin{align}
&\E[V(k+1)] \notag\\
&= \sum_{i=1}^4 \E[V_{i}(k+1)] \notag\\
&\le V(k) - \|\boldsymbol{x}(k)\|_{h\alpha \boldsymbol{L} - \frac{h}{2}\boldsymbol{K} - \frac{3h^2\alpha^2}{2}\boldsymbol{L}^2 - \frac{h}{2}(1+3h)L_f^2\boldsymbol{K}}^2 \notag\\
&\quad + \|\boldsymbol{x}(k)\|_{h^2\beta(\beta \boldsymbol{L}+\alpha \boldsymbol{L}^2)}^2 + \|\boldsymbol{x}(k)\|_{h(\beta \boldsymbol{L}+\frac{1}{2}\boldsymbol{K})+h^2(\frac{\alpha^2}{2}-\alpha\beta+\beta^2)\boldsymbol{L}^2+\frac{h}{2}(1+2h)L_f^2\boldsymbol{K}}^2 + \|\boldsymbol{x}(k)\|_{\frac{h}{2}L_f^2 \boldsymbol{K}}^2 \notag\\
&\quad - \|\boldsymbol{v}(k) + \frac{1}{\beta}\boldsymbol{g}^0(k)\|_{h(\beta-\frac{1}{2}-\frac{h}{2}-\frac{h\beta^2}{2})\boldsymbol{K} - \frac{h}{2\beta}\boldsymbol{Q} - \frac{h}{2\beta}\frac{\alpha}{\beta}\boldsymbol{K}}^2  + \|\boldsymbol{v}(k) + \frac{1}{\beta}\boldsymbol{g}^0(k)\|_{\frac{3h^2\beta^2}{2}\boldsymbol{K}}^2 \notag\\
&\quad - \left( \frac{h}{4}(1-2h L_f) - h \left(\frac{h}{\beta^2} + \frac{h}{2} + \frac{1}{2\beta^2} \right) L_f^2 - h \left( \frac{3h}{2\beta^2} + \frac{1}{2\beta} \right)  \left( \frac{1}{\rho_2(L)} + \frac{\alpha}{\beta} \right)  L_f^2 \right) \|\bar{\boldsymbol{g}}(k)\|^2 \notag\\
&\quad - \frac{h\nu}{2n}(f(\bar{\boldsymbol{x}}(k)) - f^*) + \mathcal{C}_{noise} \notag\\
&= V(k) - \|\boldsymbol{x}(k)\|_{h \boldsymbol{M}_1 - h^2 \boldsymbol{M}_2}^2 - \left\|\boldsymbol{v}(k) + \frac{1}{\beta}\boldsymbol{g}^0(k)\right\|_{h \boldsymbol{M}_3 - h^2 \boldsymbol{M}_4}^2 \notag\\
&\quad - (h\epsilon_3 - h^2\epsilon_4)\|\bar{\boldsymbol{g}}(k)\|^2 - \frac{h\nu}{2n}(f(\bar{\boldsymbol{x}}(k)) - f^*) + C_A,
\label{eq:V_total_bound}
\end{align}
where the matrices $\boldsymbol{M}_1, \boldsymbol{M}_2, \boldsymbol{M}_3, \boldsymbol{M}_4$ are defined by
\begin{align}
\boldsymbol{M}_1 &= (\alpha - \beta)\boldsymbol{L} - \frac{1}{2}(2+3L_f^2)\boldsymbol{K}, \notag\\
\boldsymbol{M}_2 &= \beta^2 \boldsymbol{L} + (2\alpha^2 + \beta^2)\boldsymbol{L}^2 + \frac{5}{2}L_f^2\boldsymbol{K}, \notag\\
\boldsymbol{M}_3 &= (\beta - \frac{1}{2} - \frac{\alpha}{2\beta^2})\boldsymbol{K} - \frac{1}{2\beta}\boldsymbol{Q}, \notag\\
\boldsymbol{M}_4 &= (2\beta^2 + \frac{1}{2})\boldsymbol{K}, \notag
\end{align}
and $C_A$ collects all variance terms arising from stochastic gradients and quantization as
\begin{align}
C_A &= 
h^2 \alpha^2\rho^2(\boldsymbol{L})\frac{nd\Delta^2}{4} + h^2n\sigma^2  +h\left( \frac{3 h} {2 \beta^2} + \frac{1}{2 \beta}\right)\left( \frac{1}{\rho_2(L)} + \frac{\alpha}{\beta} \right)\left( 2L_f^2 n \sigma^2\right) \notag\\
&\quad + h^2\beta^2\left( \frac{1}{\rho_2(\boldsymbol{L})} + \frac{\alpha}{\beta} \right)\rho^2(\boldsymbol{L})\frac{nd\Delta^2}{4} + L_f h^2 n \sigma^2  + h\left(\frac{h}{\beta^2} + \frac{h}{2} + \frac{1}{2\beta^2} \right) \left( 2L_f^2 n \sigma^2\right) \notag\\
&+ \frac{h^2 \beta}{2} (n\sigma^2 + \rho^2(\boldsymbol{L})\frac{nd\Delta^2}{4}) \notag \\
&=: h \tilde{C}_A.
\end{align}

In the following step, we obtain a recursion of $\E[V(k+1)]$ from~\eqref{eq:V_total_bound}. Here, we utilize the parametric conditions of $\alpha, \beta, h$ from the theorem statement.
From $\beta+\kappa_{1}\le\alpha$ and $\kappa_{1}=\frac{1}{\rho_{2}(\boldsymbol{L})}(4+\frac{3}{2}L_{f}^{2})$, we have
\begin{align}
(\alpha-\beta)\rho_{2}(\boldsymbol{L})-\frac{1}{2}(2+3L_{f}^{2})\ge1. \label{eq:cond_1}
\end{align}
From $\beta\ge\kappa_{4}$, we have
\begin{align}
(\beta-\frac{1}{2}-\frac{\kappa_{2}}{2\beta})-\frac{1}{2\beta\rho_{2}(\boldsymbol{L})}\ge1. \label{eq:cond_2}
\end{align}
From $\alpha\le\kappa_{2}\beta$ and $\beta\ge\kappa_{5}$, we have
\begin{align}
\epsilon_{3} &=\frac{1}{4}-\frac{1}{2\beta}\left(\frac{1}{\beta}+\frac{1}{\rho_{2}(\boldsymbol{L})}+\frac{\alpha}{\beta}\right)L_{f}^{2} \notag\\
&\ge\frac{1}{4}-\frac{1}{2\beta}\left(\frac{1}{\beta}+\frac{1}{\rho_{2}(\boldsymbol{L})}+\kappa_{2}\right)L_{f}^{2}\ge\frac{1}{8}. \label{eq:epsilon_3_bound}
\end{align}
From \eqref{eq:epsilon_3_bound}, and $0<h<\frac{\epsilon_{3}}{\epsilon_{4}}$, we have
\begin{align}
h\epsilon_{3}-h^{2}\epsilon_{4}>0. \label{eq:eta_cond}
\end{align}
From \eqref{eq:cond_1}, \eqref{eq:cond_2}, and $\alpha\le\kappa_{2}\beta$, we have
\begin{align}
\boldsymbol{M}_{1} &=(\alpha-\beta)\boldsymbol{L}-\frac{1}{2}(2+3L_{f}^{2})\boldsymbol{K} \notag\\
&\ge(\alpha-\beta)\rho_{2}(\boldsymbol{L})\boldsymbol{K}-\frac{1}{2}(2+3L_{f}^{2})\boldsymbol{K}\ge \boldsymbol{K}, \label{eq:M1_bound}\\
\boldsymbol{M}_{2} &=\beta^{2}\boldsymbol{L}+(2\alpha^{2}+\beta^{2})\boldsymbol{L}^{2}+\frac{5}{2}L_{f}^{2}\boldsymbol{K}\le\epsilon_{2}\boldsymbol{K}, \label{eq:M2_bound}\\
\boldsymbol{M}_{3} &=(\beta-\frac{1}{2}-\frac{\alpha}{2\beta^{2}})\boldsymbol{K}-\frac{1}{2\beta}\boldsymbol{Q} \notag\\
&\ge(\beta-\frac{1}{2}-\frac{\kappa_{2}}{2\beta})\boldsymbol{K}-\frac{1}{2\beta\rho_{2}(\boldsymbol{L})}\boldsymbol{K}\ge \boldsymbol{K}, \label{eq:M3_bound}\\
\boldsymbol{M}_{4} &=(2\beta^{2}+\frac{1}{2})\boldsymbol{K}\le\epsilon_{2}\boldsymbol{K}. \label{eq:M4_bound}
\end{align}
Denote $\hat{V}(k)=\|\boldsymbol{x}(k)\|_{\boldsymbol{K}}^{2}+\|\boldsymbol{v}(k)+\frac{1}{\beta}\boldsymbol{g}(k)^{0}\|_{\boldsymbol{K}}^{2}+f(\bar{x}(k))-f^{*}$.
Then, upon substituting in~\eqref{eq:V_total_bound} from~\eqref{eq:M1_bound}-\eqref{eq:M4_bound}, and with the notation $\epsilon_1 = \min\{1, \frac{\nu}{2n}\}$, we have
\begin{align}
\E[V(k+1)] \le V(k)-h(\epsilon_{1}-h\epsilon_{2})\hat{V}(k) + C_A. \label{eq:V_recurrence}
\end{align}
From the Cauchy-Schwarz inequality, we further have
\begin{align}
\epsilon_{6}\hat{V}(k)\le V(k)\le\epsilon_{5}\hat{V}(k), \label{eq:V_equivalence}
\end{align}
where $\epsilon_{6}=\min\{\frac{1}{2\rho(\boldsymbol{L})},\frac{\alpha-\beta}{2\alpha}\}$.
Since $0<h<\frac{\epsilon_{1}}{\epsilon_{2}}$, the coefficient of $\hat{V}(k)$ in~\eqref{eq:V_recurrence} is negative. 


Substituting~\eqref{eq:V_equivalence} in~\eqref{eq:V_recurrence} and $C_A = h \tilde{C}_A$, we have
\begin{align}
\E[V(k+1)] &\le (1 - h c_{1})\E[V(k)] + h \tilde{C}_A,
\label{eq:pl_const_recurrence}
\end{align}
where $c_{1} := \frac{\epsilon_{1}-h\epsilon_{2}}{\epsilon_5}$. Since $h < \frac{\epsilon_1}{\epsilon_2}$ and $\epsilon_1 <= 2\sqrt{\epsilon_{2} \epsilon_{5}}$, the coefficient $(1-h c_1) \in [0,1)$. So, we unroll the recursion~\eqref{eq:pl_const_recurrence} from $k$ to $0$, to obtain
\begin{align}
\E[V(k)] &\le (1 - h c_{1})^{k}V(0) + \frac{\tilde{C}_A}{c_{1}}.
\label{eq:pl_const_convergence}
\end{align}

Denoting $C := \frac{\tilde{C}_A}{c_{1}}$, as $k \to \infty$, the transient term vanishes (since $0 <= (1 - h c_{1}) < 1$), and we have
\[
\limsup_{k \to \infty} \E[V(k)] \le C.
\]
Hence, using the relation $\epsilon_6 \hat{V}(k) \le V(k) \le \epsilon_5 \hat{V}(k)$, for all $i\in[n]$ we have
\begin{align}
\limsup_{k \to \infty} \left(\E[\|x_{i,k}-\bar{x}_{k}\|^{2}] + \E[f(\bar{x}_{k})-f^{*}]\right)
&\le \limsup_{k \to \infty} \left(\E[\|\boldsymbol{x}_{k}\|_{\boldsymbol{K}}^{2}] + \E[f(\bar{x}_{k})-f^{*}]\right) \notag\\
&\le \limsup_{k \to \infty} \E[\hat{V}_{k}] \notag \\
& \le \limsup_{k \to \infty} \frac{1}{\epsilon_{6}}\E[V_{k}] \notag\\
&\le \frac{C}{\epsilon_{6}}.
\end{align}
The proof is complete.

%% file: sections/complexity.tex
\section{Network scaling of $\eta$, $\epsilon_i$, and the residual constants $A,B$}\label{app:AB_scaling}

Recall that $L$ is the Laplacian with $\rho(L)=\lambda_{\max}(L)$ and $\rho_2(L)=\lambda_2(L)$.
In the fixed step-size regime for Theorem~\ref{thm:sq_pdgd_convergence}, we defined in Section~\ref{sec:conv}:
\begin{align}
\epsilon_1 &:= \max\!\left\{\frac{1}{\nu_1}\Big(\frac{L_f^2}{2\beta}+\beta\,\rho(L)\Big),\;\frac12\right\},\label{eq:eps1_def_app}\\
\epsilon_2 &:= \min\!\left\{\frac{\beta}{2},\;\epsilon_1\nu_1\right\},\label{eq:eps2_def_app}\\
\epsilon_3 &:= \max\!\left\{\epsilon_1+\frac12,\;\frac{\epsilon_1}{\rho_2(L)}+\frac{\alpha}{2\beta}+\frac12\right\},\label{eq:eps3_def_app}\\
\epsilon_4 &:= \min\!\left\{\epsilon_1-\frac12,\;\frac{\alpha}{2\beta}-\frac12\right\},\label{eq:eps4_def_app}\\
\epsilon_5 &:= \max\!\left\{\beta^2\rho^2(L)+3\alpha^2\rho^2(L)+3L_f^2,\;3\beta^2\right\},\label{eq:eps5_def_app}\\
\eta &:= \sqrt{2}\,\max\!\left\{\frac{2\epsilon_1}{\rho_2(L)}+\alpha+1,\;4\epsilon_1+1\right\},\label{eq:eta_def_app}
\end{align}
and the residual constants:
\begin{align}
A &= \Bigg(\frac{\epsilon_3\eta h}{\,2\epsilon_2\epsilon_4-h\eta\epsilon_3\epsilon_5\,}\Bigg)\,
\frac{nd}{4}\,(2\alpha^2+\beta^2)\,\rho^2(L),\label{eq:A_def_app}\\
B &= \Bigg(\frac{\epsilon_3\eta h}{\,2\epsilon_2\epsilon_4-h\eta\epsilon_3\epsilon_5\,}\Bigg)\,
2n.\label{eq:B_def_app}
\end{align}

We treat the problem parameters $(L_f,\nu)$ as $\bigo(1)$ and focus on understanding the influence of network parameters, specifically the network scale $n$ and the graph Laplacian spectral properties $(\rho(L),\rho_2(L))$, on the residual error. Recall from Lemma~\ref{lem:rsi_augmented} that the effective RSI constant is given by
\[
\nu_1 = \min \left\{ \frac{\nu}{2n}, \alpha\rho_2(L) - \frac{2nL_f^2 + \nu L_f}{\nu} \right\}.
\]
The convergence condition in Theorem~\ref{thm:sq_pdgd_convergence} requires $\alpha > \frac{2nL_f^2 + \nu L_f}{\nu\rho_2(L)}$. Note that the lower bound for $\alpha$ scales as $\bigo(n/\rho_2(L))$. We select $\alpha$ to be strictly larger than this lower bound while maintaining the same asymptotic order. Specifically, we set $\alpha = \bigo(n/\rho_2(L))$ with a sufficiently large constant factor such that the second term in the definition of $\nu_1$ scales as $\bigo(n)$. Since the first term scales as $\bigo(1/n)$, for sufficiently large $n$, the minimum is determined by the first term. Consequently, we proceed with the analysis assuming $\nu_1 = \frac{\nu}{2n} = \bigo(1/n)$, alongside $\beta = \bigo(1)$ and $\alpha = \bigo(n/\rho_2(L))$.

From~\eqref{eq:eps1_def_app} with $\nu_1=\bigo(\frac{1}{n})$ and $\beta=\bigo(1)$,
\begin{equation*}
\epsilon_1
=\bigo\!\Big(\max\!\left\{\tfrac{1}{\nu_1}(1+\rho(L)),\;\frac12\right\}\Big),
\end{equation*}
implying that
\begin{equation}\label{eq:eps1_scaling_app}
\epsilon_1
=\bigo\!\Big(\max\!\left\{n(1+\rho(L)),\;\frac12\right\}\Big)
=\bigo\!\big(n\rho(L)\big).
\end{equation}
Then $\epsilon_1\nu_1=\bigo(1+\rho(L))$. So~\eqref{eq:eps2_def_app} implies
\begin{equation}\label{eq:eps2_scaling_app}
\epsilon_2=\bigo\!\big(\min\!\left\{1,\;1+\rho(L)\right\}\big) = \bigo(1).
\end{equation}
Similarly, using the definitions in~\eqref{eq:eps3_def_app},~\eqref{eq:eta_def_app},~\eqref{eq:eps5_def_app}, along with~\eqref{eq:eps1_scaling_app}, we have the rest of the constants as
\begin{align}
\epsilon_3
&=\bigo\!\left(\frac{\epsilon_1}{\rho_2(L)}+\alpha\right)
=\bigo\!\left(\frac{n\rho(L)}{\rho_2(L)}\right),\label{eq:eps3_scaling_app}\\
\eta
&=\bigo\!\left(\frac{\epsilon_1}{\rho_2(L)}+\alpha+\epsilon_1\right)
=\bigo\!\left(\frac{n\rho(L)}{\rho_2(L)}\right),\label{eq:eta_scaling_app}\\
\epsilon_5
&=\bigo\!\left(\alpha^2\rho^2(L)\right)
=\bigo\!\left(\frac{n^2\rho^2(L)}{\rho_2^2(L)}\right).\label{eq:eps5_scaling_app}
\end{align}
Consequently, substituting the asymptotic bounds for $\epsilon_3$, $\epsilon_5$, and $\eta$ yields
\begin{equation}\label{eq:eps3eta_scaling_app}
\eta \epsilon_3
=\bigo\!\left(\frac{n^2\rho^2(L)}{\rho_2^2(L)}\right), \quad \eta \epsilon_3 \epsilon_5
=\bigo\!\left(n^4 \frac{\rho^4(L)}{\rho_2^4(L)}\right).
\end{equation}
For $\epsilon_4 = \min\{\epsilon_1-\frac12,\;\frac{\alpha}{2\beta}-\frac12\}$, with $\alpha = \bigo(n/\rho_2(L))$, $\beta=\bigo(1)$, and~\eqref{eq:eps1_scaling_app}, we have 
\begin{equation}\label{eq:eps4_scaling_app}
\epsilon_4 = \bigo\left(\frac{n}{\rho_2(L)}\right).
\end{equation}
Let $R$ be the stepsize-dependent factor appearing in $A$ and $B$ in~\eqref{eq:A_def_app}-\eqref{eq:B_def_app}:
\[
R \;:=\; \frac{\epsilon_3\eta h}{\,2\epsilon_2\epsilon_4-h\eta\epsilon_3\epsilon_5\,}.
\]

The stability condition in Theorem~\ref{thm:sq_pdgd_convergence} requires the stepsize to satisfy $h < \min \left\{\frac{2 \epsilon_2 \epsilon_4}{\eta \epsilon_3 \epsilon_5}, \frac{2\epsilon_3}{\epsilon_2}\right\}$. From the above expressions~\eqref{eq:eps1_scaling_app}-\eqref{eq:eps4_scaling_app}, we observe that
\begin{align*}
    \frac{2 \epsilon_2 \epsilon_4}{\eta \epsilon_3 \epsilon_5} = \bigo\left(\frac{\rho_2^3(L)}{n^3 \rho^4(L)}\right), \, \frac{2\epsilon_3}{\epsilon_2} = \bigo\left(\frac{n \rho(L)}{\rho_2(L)}\right).
\end{align*}
So, the minimum of these two quantities is primarily decided by $\frac{2 \epsilon_2 \epsilon_4}{\eta \epsilon_3 \epsilon_5}$. Having established this, below we will consider two regimes for stepsize $h$: one where $h$ is small enough and the other where $h$ is near its maximal permissible value for stability of~\eqref{eq:update_law}.

\paragraph{Case-I.} Given the rapid growth of $\eta\epsilon_3\epsilon_5$ with respect to the network condition number $\frac{\rho(L)}{\rho_2(L)}$ and network scale $n$, we proceed with the analysis in the small stepsize regime where $h\eta\epsilon_3\epsilon_5 \ll 2\epsilon_2\epsilon_4$. This ensures that the first-order descent terms in the Lyapunov analysis dominate the higher-order discretization errors (see~\eqref{eq:vk_recur_rsi}).
Under the small stepsize regime $h\eta\epsilon_3\epsilon_5 \ll 2\epsilon_2\epsilon_4$, the denominator is dominated by $2\epsilon_2\epsilon_4 - h\eta\epsilon_3\epsilon_5 \approx 2\epsilon_2\epsilon_4$.
From~\eqref{eq:eps2_scaling_app}, $\epsilon_2 = \bigo(1)$. Substituting these relations into the definition of $R$, we obtain:
\begin{equation}\label{eq:R_scaling_app}
R \approx \frac{\epsilon_3 \eta h}{2\epsilon_2\epsilon_4} = \bigo\left(\frac{\epsilon_3\eta h}{\epsilon_4}\right) = \bigo\!\left(h\,\frac{n\rho^2(L)}{\rho_2(L)}\right).
\end{equation}

Also, with $\alpha=\bigo(n/\rho_2(L))$ and $\beta=\bigo(1)$,
\begin{equation}\label{eq:alpha2_term_app}
(2\alpha^2+\beta^2)=\bigo\!\left(\frac{n^2}{\rho_2^2(L)}\right).
\end{equation}
Substituting \eqref{eq:R_scaling_app}--\eqref{eq:alpha2_term_app} into definitions of $A,B$ in~\eqref{eq:A_def_app}-\eqref{eq:B_def_app} yields
\begin{align}
A
&= \bigo\!\left(
\Big(h\,\frac{n\rho^2(L)}{\rho_2(L)}\Big)\cdot nd\cdot
\Big(\frac{n^2}{\rho_2^2(L)}\Big)\cdot \rho^2(L)
\right)
=\bigo\!\left(h\,d\,\frac{n^4\rho^4(L)}{\rho_2^3(L)}\right),\label{eq:A_final_app}\\
B
&= \bigo\!\left(
\Big(h\,\frac{n\rho^2(L)}{\rho_2(L)}\Big)\cdot n
\right)
=\bigo\!\left(h\,\frac{n^2\rho^2(L)}{\rho_2(L)}\right).\label{eq:B_final_app}
\end{align}

\paragraph{Case-II.} Now we consider the maximal stable stepsize regime, i.e., $h=\bigo\!\Big(\frac{\epsilon_2 \, \epsilon_4}{\eta\,\epsilon_3\,\epsilon_5}\Big)$. Upon substituting from~\eqref{eq:eps3eta_scaling_app},
\begin{equation}\label{eq:h_choice_app}
h=\bigo\!\Big(\frac{\epsilon_2 \, \epsilon_4}{\eta\,\epsilon_3\,\epsilon_5}\Big)
=\bigo\!\Big(\frac{\rho_2^3(L)}{n^3\rho^4(L)}\Big),
\end{equation}
Finally, substituting from~\eqref{eq:h_choice_app} into the definitions~\eqref{eq:A_final_app}-\eqref{eq:B_final_app} gives
\begin{align}
A
&=\bigo\!\left(
\Big(\frac{\rho_2^3(L)}{n^3\rho^4(L)}\Big)\cdot d\,\frac{n^4\rho^4(L)}{\rho_2^3(L)}
\right)
=\bigo\!\left(\frac{d\,n}{1}\right)
=\bigo\!\left(d\,n\right),\label{eq:A_after_h_app}\\
B
&=\bigo\!\left(
\Big(\frac{\rho_2^3(L)}{n^3\rho^4(L)}\Big)\cdot \frac{n^2\rho^2(L)}{\rho_2(L)}
\right)
=\bigo\!\left(\frac{\rho_2^2(L)}{n\,\rho^2(L)}\right).\label{eq:B_after_h_app}
\end{align}
Equations~\eqref{eq:A_final_app}-\eqref{eq:B_final_app} and~\eqref{eq:A_after_h_app}-\eqref{eq:B_after_h_app} are exactly the scalings quoted in Remark~\ref{rem:AB_network}.


\section{Additional plots}\label{app:plots}

Under $L_f$-smoothness (and, when applicable, the PL inequality), the objective suboptimality $\mathbb{E}[F(\mathbf{x}) - F(\mathbf{x}^*)]$ is proportional to the squared distance to the solution $\mathbb{E}[\|\mathbf{x} - \mathbf{x}^*\|^2]$. Consequently, the linear dependence on the noise variances $\Delta^2$ and $\sigma^2$ predicted by our bounds manifests directly in the steady-state objective gaps.

In particular, from Theorem~\ref{thm:sq_pdgd_convergence}, the steady-state neighborhood radius under a constant step-size is limited by the quantization resolution $\Delta$ and the stochastic gradient variance $\sigma^2$. Figure \ref{fig:sensitivity} explicitly demonstrates the sensitivity of the asymptotic objective gap to these parameters across varying network sizes. The results exhibit a strict linear dependence on both the squared quantization step $\Delta^2$ and the noise variance $\sigma^2$. This empirical linear trend directly validates the derived theoretical bound from Theorem~\ref{thm:sq_pdgd_convergence}, which bounds the error proportionally to $A\Delta^2 + B\sigma^2$. Furthermore, Figure \ref{fig:sensitivity} captures the explicit network scaling effects established in Remark \ref{rem:AB_network}. In the left subfigure, the slope of the error with respect to $\Delta^2$ increases as the number of agents $n$ grows (with $\sigma^2$ fixed at 0.5). This aligns perfectly with the theoretical scaling of $A$, which depends linearly on the number of agents $n$. Conversely, in the right subfigure, the slope of the error with respect to $\sigma^2$ decreases as $n$ increases (with number of communication bits fixed at 14). This clearly corroborates the derived theoretical behavior where $B$ is inversely proportional to $n$, demonstrating the variance-reduction benefit of larger networks.

\begin{figure}[htpb]
    \centering
    \includegraphics[width=\linewidth]{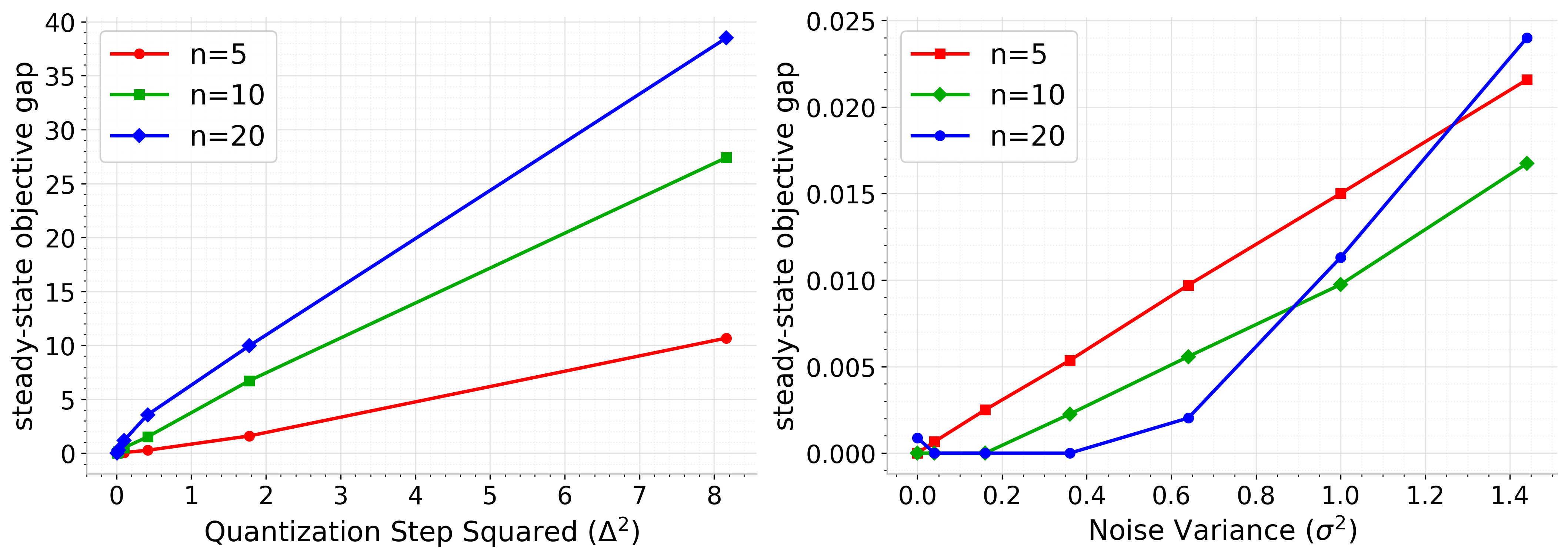}
    \caption{Asymptotic objective gap of \qpdgd{} for Example~1 versus (left) quantization step squared $\Delta^2$ and (right) gradient-noise variance $\sigma^2$, for multiple network sizes $n$. The linear trends corroborate the $\mathcal{O}(A\Delta^2 + B\sigma^2)$ steady-state bound from Theorem~\ref{thm:sq_pdgd_convergence} and the network scaling discussed in Remark~\ref{rem:AB_network}.}
    \label{fig:sensitivity}
\end{figure}

\begin{figure}[htpb]
    \centering
    \includegraphics[width=\linewidth]{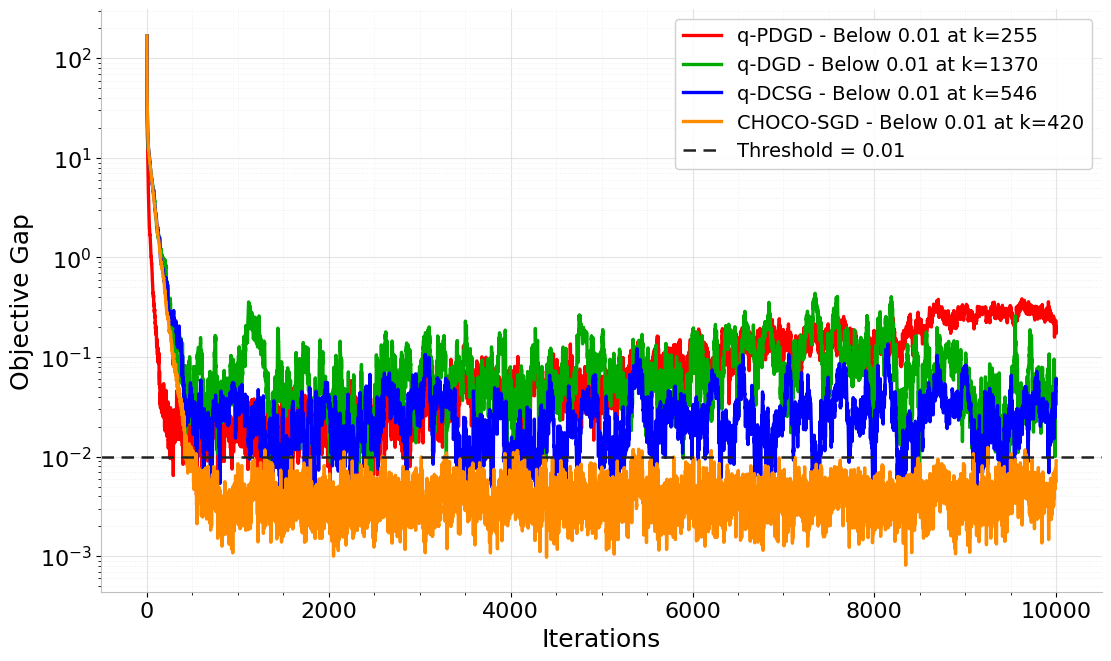}
    \caption{Iterations to reach a target objective threshold: \qpdgd{} versus \texttt{q-DGD} and \texttt{DCSG} on distributed least squares. \qpdgd{} requires markedly fewer iterations to reach the target threshold.}
    \label{fig:baseline_comparison}
\end{figure}

%% file: sections/ML.tex
\section{Empirical Evaluation on Distributed Deep Learning}
\label{app:dl-experiments}

The theoretical results in Section~\ref{sec:proofs} are stated under RSI (Assumption~3)
or PL (Assumption~4) on the global objective. Modern deep networks satisfy neither
condition globally: the loss landscape is non-convex with many saddle points and local
minima. Nevertheless, a practical algorithm motivated by these analyses should remain
robust in this regime, since (i)~the consensus and dual-tracking mechanism is
geometry-agnostic and (ii)~PL-like behavior is widely observed locally near
minima of overparameterized networks. We therefore stress-test \texttt{q-PDGD}
against representative quantized decentralized baselines on two image-classification
benchmarks where our assumptions are formally violated, and show that the
consensus advantage predicted by Theorem~\ref{thm:diminishing_step} continues to hold.

\paragraph{Setup.}
We train a small convolutional network collaboratively across $n=100$ agents
arranged in a ring topology, communicating once per local SGD step over a
$4$-bit stochastic uniform quantizer on the range $[-1, 1]$ (Eq.~\eqref{lem:quant}).
The model is a two-layer CNN ($1{\to}8{\to}16$ channels with $5{\times}5$ kernels and
$2{\times}2$ max-pooling, followed by a linear classifier) totaling $\approx\!11{,}000$
parameters per agent.
We evaluate on \textbf{MNIST} and \textbf{Fashion-MNIST}; the training set is
shuffled and split IID across the $100$ agents ($600$ images each), so all heterogeneity
in the trajectories arises from local stochastic gradients and quantization noise rather
than from data heterogeneity. Each agent draws a minibatch of size $32$ and runs for
$T=500$ communication rounds. All agents are initialized identically, and the random
seed (controlling minibatches and quantization noise) is reset before each algorithm,
so any difference between curves is attributable solely to the consensus mechanism.


\paragraph{Baselines and hyperparameters.}
We compare \texttt{q-PDGD} against \texttt{q-DGD}~\citep{Yuanetal2013},
\texttt{q-DCSG}~\citep{dutta20241}, and \texttt{CHOCO-SGD}~\citep{koloskova2019decentralized}.
The mixing matrix is $W = I - \omega L$ with $\omega = 1/(\Delta_{\max}+1)$, where
$\Delta_{\max}$ is the maximum degree.
For a fair comparison the gradient stepsize is fixed to $h=0.05$ across all algorithms,
so the comparison isolates the consensus dynamics.
\texttt{q-PDGD} uses $\alpha = 0.5$, $\beta = 0.5$;
\texttt{q-DCSG} uses diminishing $\alpha_k = 0.05/(1+0.005\,k)$,
$\beta_k = 0.5/(1+0.005\,k)$;
\texttt{CHOCO-SGD} uses consensus stepsize $\gamma = 0.5$.

\paragraph{Metrics.}
At each evaluation round we record (i)~per-agent test accuracy
$\{\mathrm{acc}(x_i)\}_{i=1}^n$, (ii)~the consensus error
$\tfrac{1}{n}\sum_i\|x_i - \bar x\|$, (iii)~mean training loss across agents, and
(iv)~the test accuracy of the \emph{averaged-parameter} model
$\bar x = \tfrac{1}{n}\sum_i x_i$, which is the model that would actually
be deployed.

\paragraph{Results.}
Figures~\ref{fig:mnist-main}--\ref{fig:fashion-disp} report the comparison.
Three observations are consistent across both datasets and align with our theory.

\emph{(1) Smallest consensus residual.}
Panel~(b) of Figs.~\ref{fig:mnist-main} and \ref{fig:fashion-main} (log scale) shows
that \texttt{q-PDGD} stabilizes at a consensus error roughly $4{-}8\times$ below
\texttt{CHOCO-SGD} and an order of magnitude below \texttt{q-DGD}/\texttt{q-DCSG}.
This is the structural advantage predicted by the dual variable $v_k$ in
Eq.~\eqref{eq:update_law}: it accumulates and feeds back disagreement induced by
quantized mixing, suppressing the residual neighborhood radius even when the
quantizer is aggressive ($4$ bits).

\emph{(2) Best deployed (consensus) model.}
Panel~(d) shows that the averaged model from \texttt{q-PDGD} dominates throughout
training. Tighter consensus means $\bar x$ remains close to every $x_i$, so
averaging does not destroy useful local progress; on Fashion-MNIST,
\texttt{q-DGD} loses several accuracy points purely to disagreement.

\emph{(3) Tight cross-agent distribution.}
Figures~\ref{fig:mnist-disp} and \ref{fig:fashion-disp} report the spread of
final per-agent accuracies. On MNIST, \texttt{q-PDGD} reaches
$\mu = 0.966$, $\sigma = 0.0050$, against \texttt{q-DGD}'s $\mu = 0.906$,
$\sigma = 0.0277$; on Fashion-MNIST, $\mu = 0.819$, $\sigma = 0.0158$ against
$\mu = 0.749$, $\sigma = 0.0205$. The histograms are visibly concentrated for
\texttt{q-PDGD}, indicating that no agent is left behind.

Although the global-geometry assumptions of Section~\ref{sub:ass} do not hold
for these non-convex CNN objectives, the qualitative behavior predicted by our
analysis---in particular, a markedly smaller asymptotic consensus neighborhood
than competing quantized decentralized methods---persists, suggesting that the
primal--dual disagreement-tracking mechanism is useful well beyond the regime
covered by the formal guarantees.

\begin{figure}[t]
    \centering
    \includegraphics[width=\linewidth]{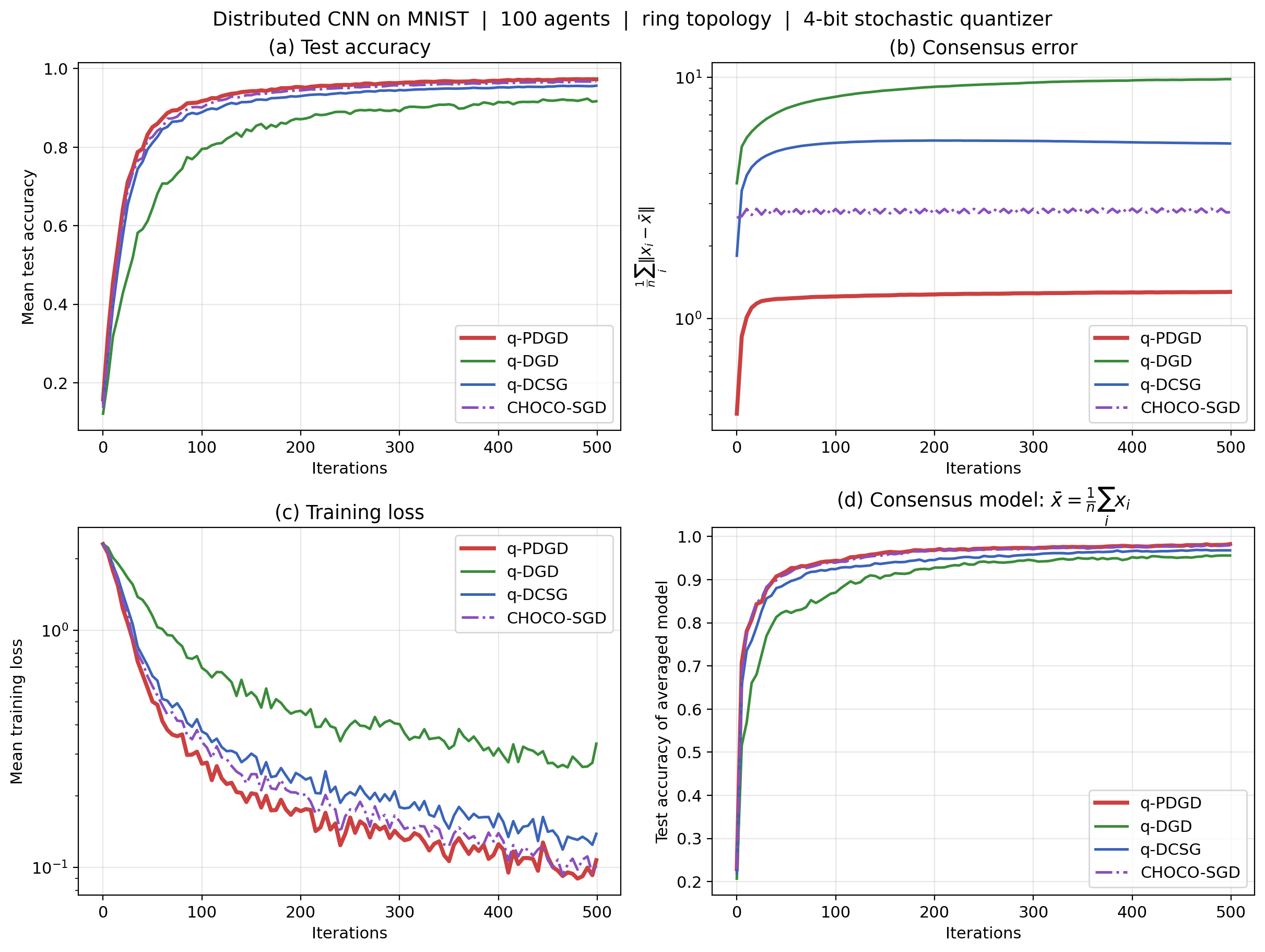}
    \caption{\textbf{MNIST, $100$ agents, ring topology, $4$-bit stochastic quantizer.}
    (a)~Mean per-agent test accuracy. (b)~Consensus error $\tfrac{1}{n}\sum_i\|x_i-\bar x\|$
    (log scale): \texttt{q-PDGD} is roughly an order of magnitude lower than the baselines.
    (c)~Mean training loss. (d)~Test accuracy of the averaged (deployed) model
    $\bar x = \tfrac{1}{n}\sum_i x_i$.}
    \label{fig:mnist-main}
\end{figure}

\begin{figure}[t]
    \centering
    \includegraphics[width=\linewidth]{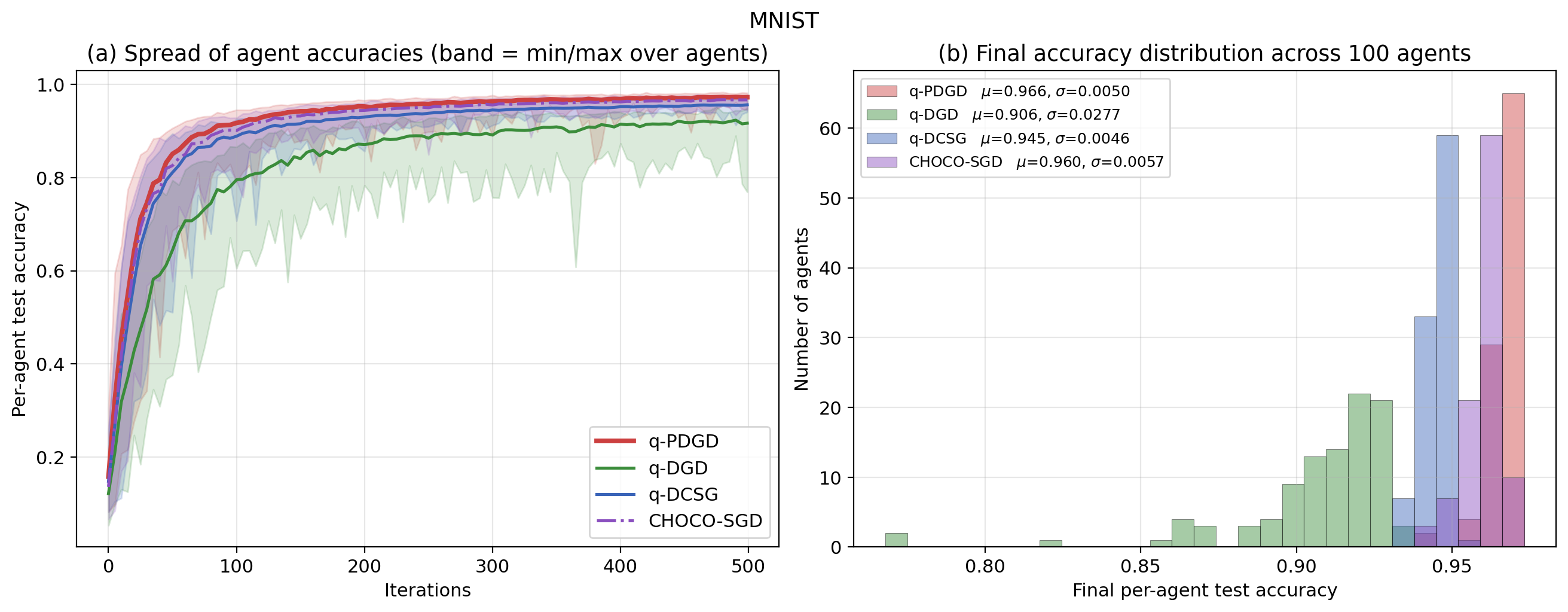}
    \caption{\textbf{MNIST: dispersion across the $100$ agents.}
    (a)~Mean accuracy with shaded min/max envelope. (b)~Histogram of final
    per-agent test accuracies. \texttt{q-PDGD} produces the right-most and
    most concentrated distribution.}
    \label{fig:mnist-disp}
\end{figure}

\begin{figure}[t]
    \centering
    \includegraphics[width=\linewidth]{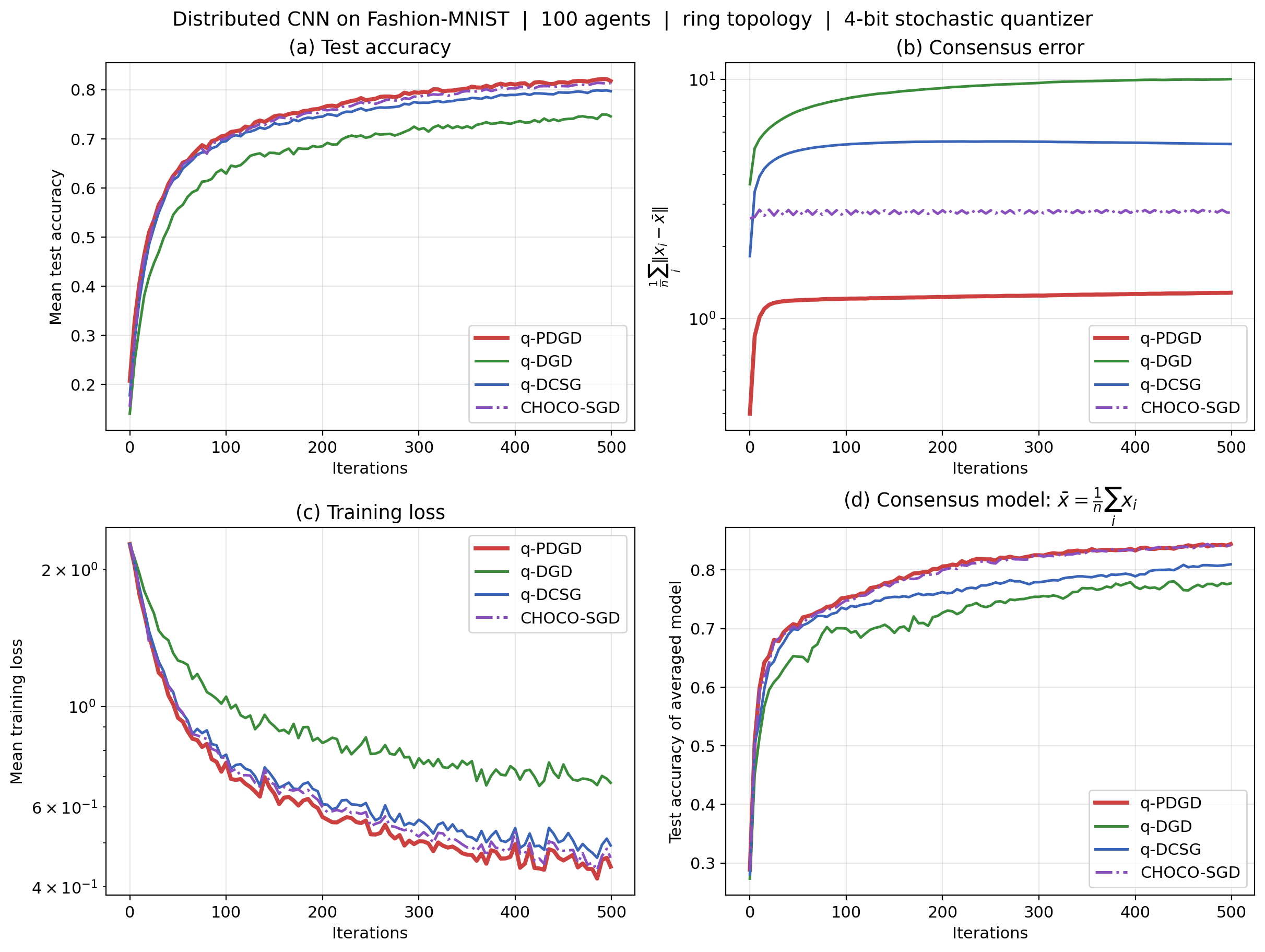}
    \caption{\textbf{Fashion-MNIST, $100$ agents, ring topology, $4$-bit stochastic
    quantizer.} Same panels as Fig.~\ref{fig:mnist-main}. The harder task amplifies
    the gap between algorithms in panel~(d): tight consensus directly translates
    into a better deployed model.}
    \label{fig:fashion-main}
\end{figure}

\begin{figure}[t]
    \centering
    \includegraphics[width=\linewidth]{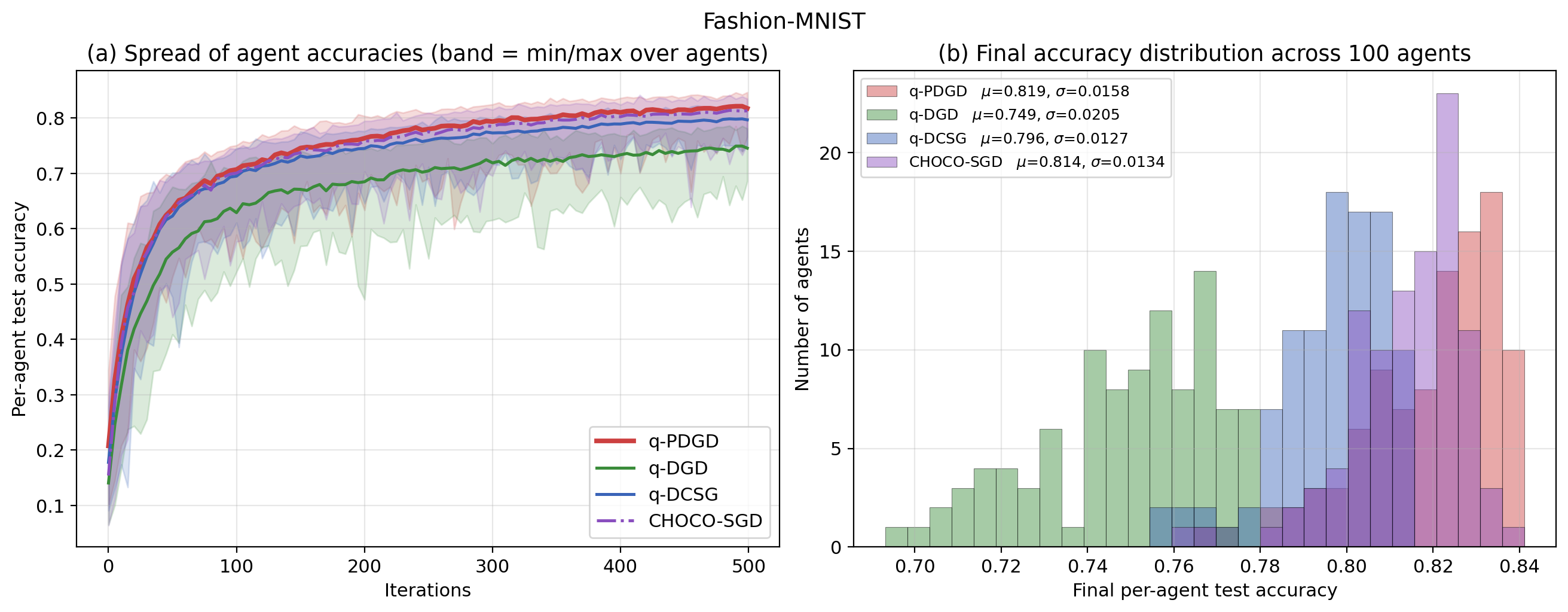}
    \caption{\textbf{Fashion-MNIST: dispersion across the $100$ agents.}
    Final per-agent accuracies: \texttt{q-PDGD} $\mu{=}0.819$, $\sigma{=}0.0158$;
    \texttt{CHOCO-SGD} $\mu{=}0.814$, $\sigma{=}0.0134$;
    \texttt{q-DCSG} $\mu{=}0.796$, $\sigma{=}0.0127$;
    \texttt{q-DGD} $\mu{=}0.749$, $\sigma{=}0.0205$.}
    \label{fig:fashion-disp}
\end{figure}